\documentclass[11pt,a4]{article}

\usepackage[utf8]{inputenc}
\usepackage[english,main=russian]{babel}
\usepackage{csquotes}

\usepackage{amsfonts, amsmath, amsthm}
\usepackage{amssymb}
\usepackage{url}
\usepackage{booktabs} 
\usepackage{comment}

\DeclareMathOperator{\minimize}{minimize}
\newcommand\br[1]{\left(#1\right)}

\newcommand\sbr[1]{\left[#1\right]}
\newcommand\mbr[1]{\left|#1\right|}

\newcommand\nbr[1]{\left\|#1\right\|}

\newcommand{\proj}{{\rm proj}}
\def\moh#1{{\color{black}#1}}
\newtheorem{example}{Пример}
\newtheorem{proposition}{Предложение}

\usepackage[breaklinks=true,colorlinks=true,linkcolor=red, citecolor=blue, urlcolor=blue]{hyperref}%

\usepackage[backend=biber,bibencoding=utf8,sorting=none,style=gost-numeric,language=autobib,autolang=other,clearlang=true,sortcites=true,doi=false,isbn=false,date=year]{biblatex}

\usepackage{mathtools}
\usepackage{makecell}
\usepackage{tikz}
\usepackage{dsfont}
\usepackage{secdot}
\usepackage{longtable}

\usepackage{algorithm}
\usepackage{algorithmic}
\usepackage{float}

\sectiondot{section}
\sectiondot{subsection}
\usepackage{titlesec}
\titleformat{\subsection}
  {\centering\normalfont\normalsize\bfseries\itshape}{\thesubsection.}{1em}{}
  
\addto\captionsrussian{}

\newtheorem{theorem}{Теорема}

\newtheorem{corollary}{Следствие}
\newtheorem{definition}{Определение}
\newtheorem{remark}{Замечание}

\DeclareMathOperator*{\argmin}{arg\,min}

\allowdisplaybreaks
\usepackage[symbol]{footmisc}

\def \R {\mathbb R}

\def \EE {\mathbb E}

\begin{document}

\textbf{УДК} 519.85

\begin{center}
\Large{\textbf{О некоторых работах Бориса Теодоровича Поляка по сходимости градиентных методов и их развитии}}\footnote{
Работа выполнена при финансовой поддержке Минобрнауки РФ (проект FSMG-2024-0011).   
}

%%%%%%%%%%%%%%%%%%%%%%%%%%%%%%%%%%%%%%%%%%%%%%%%%%%%%%%%%%%                 Порядок авторов %%%%%%%%%%%%%%%%%%%%%%%%%%%%%%%%%%%%%%%%%%%%%%%%%%%%%%%%%%%

{\bf
\copyright\,2023 г.\,\,
С.\,С.~Аблаев$^{1,4}$,
А.\,Н.~Безносиков$^{1,2}$,
А.\,В.~Гасников$^{1,2,3}$,
Д.\,М.~Двинских$^{1,2,3}$,
A.\,B.~Лобанов$^{1,3}$,
C.\,М.~Пучинин$^{1}$,
Ф.\,С.~Стонякин$^{1,4}$
} % В ЖВМ и МФ авторы указываются в алфавитном порядке

\small{
$^{1}$141700 Долгопрудный, М.о., Институтский пер., 9, НИУ МФТИ;

$^{2}$127051, г. Москва,
Б. Каретный переулок,
д.19 стр. 1, ИППИ РАН;

$^{3}$109004 Москва, ул. А. Солженицына, 25, ИСП РАН;

$^{4}$295007 Республика Крым, г. Симферополь, просп. академика Вернадского 4, Крымский федеральный университет имени В.И. Вернадского;
}
%%%%%%%%%%%%%%%%%%%%%%%%%%%%%%%%%%%%%%%%%%%%%%%%%%%%%%%%%%%

\small{e-mail: gasnikov@yandex.ru}

\small{Поступила в редакцию: 15.09.2023 г.}\\
\small{Переработанный вариант  16.12.2023 г.}\\
\small{Принята к публикации 17.11.2023 г.}
\end{center}

\renewcommand{\abstractname}{\vspace{-\baselineskip}}
\begin{abstract}
  \textbf{Аннотация:} В статье представлен обзор современного состояния субградиентных и ускоренных методов выпуклой оптимизации, в том числе при наличии помех и доступа к различной информации о целевой функции (значение функции, градиент, стохастический градиент, старшие производные). Для невыпуклых задач рассматривается условие Поляка--Лоясиевича и приводится обзор основных результатов. Рассматривается поведение численных методов при наличии острого минимума. Цель данного обзора -- показать влияние работ Б.Т. Поляка (1935 -- 2023) по градиентным методам оптимизации и их окрестностям на современное развитие численных методов оптимизации.
\end{abstract}

\textbf{Ключевые слова:} градиентный спуск, условие градиентного доминирования (Поляка--Лоясиевича), острый минимум, субградиентный метод Поляка--Шора, условие ранней остановки, метод тяжелого шарика Поляка, стохастический градиентный спуск

\section{Введение}\label{section:introduction}
Данный обзор посвящен частичному разбору нескольких работ Б.Т.~Поляка, о сходимости методов градиентного типа,  которые на многие десятилетия определили развитие численных методов оптимизации. В частности, продолжают активно цитироваться и развиваться в настоящее время. Прежде всего, речь пойдет об этих работах \cite{поляк1963градиентные,поляк1964градиентные,поляк1964некоторых,левитин1966методы,поляк1969минимизация,поляк1969метод,поляк1980оптимальные,poljak1981iterative,polyak1982sharp,поляк1983введение,немировский1985оптимальные,поляк1990оптимальные,поляк1990новый,polyak1992acceleration,nesterov2006cubic}. 

Подчеркнем, что в данной статье не планируется описывать научный путь Бориса Теодоровича. Мы коснемся лишь пары десятков статей из более чем 250. Более подробно с научным путем Б.Т. Поляка можно познакомиться, например, по статьям \cite{гасников2023научный,fradkov2023polyak}.

Структура обзора следующая. В разделе \ref{sec:nonsmooth} описываются методы негладкой оптимизации и специальный способ адаптивного выбора шага (в литературе часто называется <<шаг Поляка>>).  В частности, приводится субградиентный метод Поляка--Шора и вариант этого метода с переключением (для задач с функциональными ограничениями). Далее обсуждается вопрос о возможной линейной скорости сходимости таких методов, если минимум острый. 

В разделе \ref{sec:smooth} излагаются методы гладкой оптимизации. Начинается изложение с описания условия градиентного доминирования, которое обобщает условие сильной выпуклости целевого функционала, не предполагая при этом даже выпуклости. Показывается, что при данном условии градиентный спуск глобально линейно сходится. Рассматриваются различные обобщения данного результата. В частности, на случай, когда градиент доступен лишь с заданным уровнем относительной неточности. Затем идет изложение метода условного градиента и некоторых современных результатов вокруг этого метода. В заключение данного раздела описывается метод тяжелого шарика Поляка, который породил линейку современных ускоренных методов. 

В заключительном разделе \ref{sec:stoch} приводятся результаты Поляка--Цыпкина и Поляка--Юдицкого--Рупперта, в которых возникают различные формы центральной предельной теоремы для выхода алгоритма типа стохастического градиентного спуска. Далее приводятся неасимптотические результаты, в том числе и для ускоренных вариантов стохастического градиентного спуска. В частности, рассматриваются, так называемые, мультипликативные помехи, которые в современных исследованиях чаще называют условием сильного роста, не ассоциируя это с пионерскими работами Б.Т. Поляка с соавторами. В заключение рассматриваются рандомизированные безградиентные (также называемые, поисковые) методы. В условиях повышенной гладкости целевой функции обсуждается конструкция Поляка--Цыбакова, позволяющая строить хорошую модель производной по направлению целевой функции, исходя всего из двух проб.

\section{Негладкая выпуклая оптимизация}\label{sec:nonsmooth}
История развития методов негладкой оптимизации начинается в 60-е годы прошлого века и достаточно подробно описана в работе \cite{polyak1978subgradient}. Наряду с работами Н.З.~Шора \cite{shor2012minimization} важный вклад в развитие этой области принадлежит Б.Т. Поляку \cite{поляк1969минимизация,polyak1982sharp}. 
\subsection{Субградиентный метод Поляка--Шора}

Если не оговорено иное, то далее в тексте рассматриваются задачи вида 
\begin{equation}\label{probl}
f(x) \rightarrow \min_{x \in Q},    
\end{equation}
где $f(x)$ -- необязательно дифференцируемая выпуклая функция, $Q$ -- выпуклое замкнутое подмножество $\mathbb{R}^n$, $f^* = f(x_*) = \min_{x \in Q} f(x)$.

В данном разделе мы поговорим про субградиентные методы для негладкой оптимизации и вклад Бориса Теодоровича Поляка. Хорошо известно, что при минимизации недифференцируемых функций возникает ряд проблем: неприменим метод покоординатного спуска, а субградиент целевой фyнкции не задаёт направление наискорейшего возрастания. В связи с этим Н.~З.~Шором был предложен субградиентный метод \cite{Shor}, являющийся прямым аналогом градиентного метода. Особенность идеи такого метода в том, что вместо градиента целевой функции в методе используется произвольный сyбградиент негладкой выпyклой фyнкции. Рассмотрим случай, когда $Q$ --- замкнyтое выпyклое подмножество $\mathbb{R}^n$ с евклидовой нормой $\|\cdot\|_2$. Пусть $B_2^n(x_*, R)=\left\{ {x\in \mathbb{R}^n:\;\;\left\|{x-x_*} \right\|_2 \leqslant R} \right\}.$ Будем всюду далее обозначать субградиент $f$ (некоторый элемент субдифференциала $\partial f\left( x \right)$) в точке $x$ как $\nabla f\left( x \right)$. Если функция $f$ дифференцируема в точке $x$, то $\nabla f(x)$~--- её градиент. Итерация сyбградиентного метода при $h_k >0$ имеет следyющий вид
\begin{equation}\label{p0_eq2}
x^{k+1} = Pr_Q\{x^k - h_k \nabla f(x^k)\}, \quad \nabla f(x^k) \in \partial f(x^k),
\end{equation}
где $\operatorname{Pr}_Q (y) := \argmin_{x \in Q} \{\|y - x \|_2 \}$~--- оператор евклидова проектирования на множество $Q$. Одна из главных особенностей субградиентных методов состоит в том, что значения функции в этом методе могут не убывать монотонно с ростом количества итераций. Вообще, $f$ не обязательно убывает вдоль направления $- \nabla f(x^k)$, обратного направлению сyбградиента в текyщей точке. Однако оказывается, что при этом возможно гарантировать монотонное убывание расстояния от текущей точки до точки минимума. Вторая особенность --- это выбор шага сyбградиентного метода. Если выбирать постоянный шаг, то метод может не сходиться. Действительно, пycть для фyнкции одной переменной $f(x) = |x|$, $x_0 = -0.01$ и выбран постоянный шаг $h_k = 0.02$. Тогда итеративная последовательность метода \eqref{p0_eq2} будет состоять всего из двyх точек $-0.01$ и $0.01$ и не бyдет сходимости к точке минимума $0$.

Интересный подход к выбору шага в сyбградиентном методе предложен Б.Т. Поляком. Он предложил при выборе шага использовать степень близости значения функции в текущей точке к минимальному. Это вполне возможно, если известно искомое минимальное значение функции $f^*$. Например, если систему совместных линейных уравнений 
$$\langle a_i, x \rangle = b_i, \quad i = 1, \ldots, n, \quad x \in \mathbb{R}^n,$$
свести к минимизации функции 
$$f(x) = \sum\limits_{i=1}^n \big{|}\langle a_i, x \rangle - b_i \big{|},$$
то $f^* = 0$. Также $f^*$ бывает известно в геометрических задачах: проекция точки на множество, нахождение общей точки системы множеств. Используя $f^*$ можно построить адаптивный вариант шага (впервые он предложен в \cite{поляк1969минимизация}), не содержащий таких параметров задачи, как константа Липшица целевой функции или расстояние от точки старта до множества решений:
\begin{equation}\label{h}
h_k = \frac{f(x^k) - f^*}{\|\nabla f(x^k)\|_2^2}.
\end{equation}
Такой шаг принято называть шагом Б.Т. Поляка. Для субградиентного метода с таким шагом известен следующий результат о сходимости \cite{поляк1983введение}.
\begin{theorem}
Пусть $f(x)$ --- выпуклая на $\mathbb{R}^n$ функция, множество точек минимума $X_*$ которой не пусто. Тогда в методе \eqref{p0_eq2} с шагом \eqref{h}
%$$x_{k+1} = x_k - \frac{f(x_k) - f^*}{\|\nabla f(x_k)\|_2^2} \nabla f(x_k)$$
$x^k \rightarrow x_* \in X_*$. При этом $\lim\limits_{k \rightarrow \infty} \sqrt{k} (f(x^k) - f^*) = 0$.
\end{theorem}
\begin{proof}
Как известно, для всякой точки минимума $x_* \in X_*$ верны неравенства
$$
2 h_k (f(x^k) - f(x_*)) \leq
2h_k \langle \nabla f(x^k), x^k-x_* \rangle \leq
$$
$$
\leq h_k^2 \| \nabla f(x^k)\|_2^2 + \|x^k - x_* \|_2^2 -
\|x^{k+1} - x_* \|_2^2.
$$
Поэтому
\begin{equation}
\|x^{k+1} - x_* \|_2^2 \leq
h_k^2 \| \nabla f(x^k)\|_2^2 - 2\,  h_k (f(x^k) - f(x_*)) +
\|x^k - x_* \|_2^2=
\end{equation}
$$
= \frac{(f(x^k) - f(x_*))^2}{\| \nabla f(x^k) \|_2^2} -
\frac{2\,(f(x^k) - f(x_*))^2}{\| \nabla f(x^k) \|_2^2} +
\|x^k - x_* \|_2^2 =
$$
\begin{equation}\label{pp1}
= - \frac{(f(x^k) - f(x_*))^2}{\| \nabla f(x^k) \|_2^2} +
\|x^k - x_* \|_2^2 .
\end{equation}

Таким образом,  
$$\frac{f(x^k) - f^*}{\| \nabla f(x^k)\|_2} \rightarrow 0.$$ 
Более того, неравенство $\|x^k-x_*\|_2 \leq \|x^0-x_*\|_2$ означает, что последовательность $x^k$ ограничена, и тогда $\| \nabla f(x^k)\|_2 \leq M$. Поэтому $f(x^k) \rightarrow f^*$. Следовательно, найдётся подпоследовательность $x^{k_l} \rightarrow x_*$. Итак, получаем, что $\|x^k - x_*\|_2$ монотонно убывает, а $\|x^{k_l} - x_*\|_2 \rightarrow 0$. Отсюда $x^k \rightarrow x_*$. Ввиду \eqref{pp1} получаем
$$\sum_{k=0}^{\infty} \frac{(f(x^k) - f^*)^2}{\| \nabla f(x^k)\|_2^2} < \infty,$$
а из ограниченности $\| \nabla f(x^k)\|_2$ следует $\sum\limits_{k=0}^{\infty} (f(x^k) - f^*)^2 < \infty$. Если предположить, что $\lim\limits_{k \rightarrow \infty} \sqrt{k} (f(x^k) - f^*) > 0$, то $f(x^k) - f^* > \frac{a}{\sqrt{k}}$ для достаточно больших $k$, что противоречит условию $\sum\limits_{k=0}^{\infty} (f(x^k) - f^*)^2 < \infty$. Итак, $\lim\limits_{k \rightarrow \infty} \sqrt{k} (f(x^k) - f^*) = 0$.
\end{proof}
Предыдущий результат означает, что для достижения точности $\varepsilon>0$ решения задачи по функции гарантированно достаточно $O\left(\frac{1}{\varepsilon^2}\right)$ итераций. Данная оценка скорости сходимости неулучшаема на классе минимизационных задач с выпуклыми липшицевыми (как гладкими, так и негладкими) целевыми функциями. Хотя известны и другие подходы к выбору шага для субградиентного метода. Например, если при $Q = \mathbb{R}^n$ предположить (здесь и всюду далее $R \geqslant \left\| x^0-x_*  \right\|_2 $), что
\begin{equation}\label{p0_eq3}
    \left\|\nabla f(x) \right\|_2 \leqslant M \text{ для всякого } x\in B_2^n(x_*, R\sqrt{2}),
\end{equation}
и выбрать шаг субградиентного метода, а также точку выхода следующим образом:
\begin{equation}\label{p0_eq4}
    h=\frac{R}{M\sqrt N},
    \quad \overline{x}^N=\frac{1}{N}\sum\limits_{k=0}^{N-1} {x^k},
\end{equation}
то будет выполняться неравенство
\begin{equation}\label{p0_eq5}
    f\left(\overline{x}^N \right)-f(x_*) \leqslant \frac{MR}{\sqrt N}.
\end{equation}
Неравенство (\ref{p0_eq5}) означает, что при выборе
\begin{equation}
N=\frac{M^2R^2}{\varepsilon ^2},\quad h=\frac{\varepsilon }{M^2}
\end{equation}
будет достигаться оценка $f\left(\overline{x}^N \right)-f(x_*) \leqslant \varepsilon$, которая оптимальна с точностью до умножения на константу. Действительно, известно, что
точная нижняя оценка на классе задач выпуклой оптимизации с условием (\ref{p0_eq3}) для методов первого порядка вида
\begin{equation}
x^{k+1} \in x^0 + \text{span}\left\{\nabla f(x^0),...,\nabla f(x^k)\right\},
\end{equation}
где для всех $j = 0, 1, ..., k$ верно $\nabla f(x^j) \in \partial f(x^j)$,
имеет вид~\cite{Drori}:
\begin{equation}
f\left({x^N} \right)-f(x_*) \geqslant \frac{MR}{\sqrt {N+1} }.
\end{equation}

Представляется, что интерес выбора шага Б.Т. Поляка для сyбградиентного метода в том, что он в каждой точке позволяет учесть динамику значений целевой функции и не содержит параметров типа константы Липшица целевой фyнкции или оценки расстояния от начальной точки до множества точных решений задачи. Такого типа процедуры выбора шагов в оптимизационных методах часто называют адаптивными. Разным подходам к адаптивным процедурам при выборе шагов оптимизационных методов посвящаются всё новые работы \cite{loizou2021stochastic, wang_2023, hazan2019revisiting}. В этом плане можно отметить многочисленные исследования в области универсальных градиентных методов \cite{NesterovUniversal, гасников2018универсальный}, адаптивных методов типа AdaGrad, да и разные модификации шага Б.Т. Поляка в детерминированном и стохастическом случае \cite{loizou2021stochastic,jiangstich2023SGDPolyakstep}. Важно, что, помимо процедyры выбора шага, Б.Т. Поляк ещё выделил и класс задач с острым минимyмом \cite{поляк1969минимизация}, для которого этот шаг позволил доказать резyльтат о сходимости сyбградиентного метода со скоростью геометрической прогрессии. Далее поговорим об этом результате и его развитии в современных работах.

\subsection{Острый минимум и линейная скорость сходимости субградиентного метода с шагом Б.Т. Поляка}

Как видим, в описанных выше результатах оценки скорости сходимости субградиентного метода сублинейны. Получить линейную сходимость субградиентного метода можно лишь с помощью дополнительных предположений. Например, линейная скорость сходимости возможна для методов типа секущей гиперплоскости, применимых к задачам малой или умеренной размерности. Что касается задач большой размерности, то сходимость со скоростью геометрической прогрессии может позволить получить дополнительное предположение об остром минимуме 
\begin{equation}\label{tochka}
f(x) - f^* \geqslant \alpha\, \min_{x_* \in X_*} \|x - x_*\|_{2},
\end{equation}
где $X_*$~--- множество точек минимума функции $f$ на множестве $Q$, $\alpha >0$ --- некоторое фиксированное положительное число. В частности, условие \eqref{tochka} верно для задачи евклидова
проектирования точки $x$ на выпуклый компакт $X_{*}\subset Q$, причем $f^{*}=0$. Условие острого минимума было введено Б.Т. Поляком в \cite{поляк1969минимизация}. Рассмотрим субградиентный метод \eqref{p0_eq2} с шагом Б.\,Т.\,Поляка~\cite{поляк1969минимизация} \eqref{h} для задачи минимизации $f$.

\begin{theorem}
    Пусть $f$~--- выпуклая функция и для задачи минимизации~$f$ с острым минимумом
    используется алгоритм \eqref{p0_eq2} c шагом Б.\,Т.\,Поляка \eqref{h}.
    Тогда после $k$ итераций алгоритма~\eqref{p0_eq2} верно неравенство
    $$
    \min_{x_* \in X_*}\|x^{k+1} - x_* \|_2^2 \leq \prod_{i=0}^k \Bigl(
    1 - \frac{\alpha^2}{\| \nabla f(x^i)\|_2^2} \Bigr)
    \min_{x_* \in X_*}\|x^0 - x_* \|_2^2.
    $$
\end{theorem}

\begin{proof}
Согласно \eqref{pp1} и условию острого минимума имеем:
\begin{equation}
\min_{x_* \in X_*} \|x^{k+1} - x_* \|_2^2 \leq
- \frac{\alpha^2}{\| \nabla f(x^k)\|_2^2}
\min_{x_* \in X_*} \|x^k - x_* \|_2^2 +
\min_{x_* \in X_*}\|x^k - x_* \|_2^2 =
\end{equation}
$$
= \Bigl(
1 - \frac{\alpha^2}{\| \nabla f(x^k)\|_2^2} \Bigr)
\min_{x_* \in X_*}\|x^k - x_* \|_2^2.
$$
Далее, получаем цепочку неравенств
\begin{equation}
\min_{x_* \in X_*} \|x^{k+1} - x_* \|_2^2 \leq
\Bigl(
1 - \frac{\alpha^2}{\| \nabla f(x^k)\|_2^2} \Bigr)
\min_{x_* \in X_*}\|x^k - x_* \|_2^2 \leq
\end{equation}
$$
\leq \Bigl(
1 - \frac{\alpha^2}{\| \nabla f(x^k)\|_2^2} \Bigr)
\Bigl(
1 - \frac{\alpha^2}{\| \nabla f(x^{k-1})\|_2^2} \Bigr)
\min_{x_* \in X_*} \|x^{k-1} - x_* \|_2^2
\leq \ldots \leq
$$
\begin{equation}
\leq\prod_{i=0}^k \Bigl(
1 - \frac{\alpha^2}{\| \nabla f(x^i)\|_2^2} \Bigr)
\min_{x_* \in X_*}\|x^0 - x_* \|_2^2.
\end{equation}
\end{proof}

\begin{corollary}
    Если в условиях предыдущей теоремы допустить,
    что $f$ удовлетворяет условию Липшица с константой $M >0$ (т.е. все нормы сyбградиентов $f$ равномерно сверхy ограничены этой константой),
    то можно утверждать сходимость алгоритма~\eqref{p0_eq2}
    со скоростью геометрической прогрессии
    $$
    \min_{x_* \in X_*}\|x^{k+1} - x_* \|_2^2 \leq
    \Bigl( 1 - \frac{\alpha^2}{M^2} \Bigr)^{k+1}
    \min_{x_* \in X_*}\|x^0 - x_* \|_2^2.
    $$
\end{corollary}

\subsection{Субградиентные схемы с переключениями для задач с ограничениями}
Б.\,Т.\,Поляк внёс существенный вклад и в изучение задач математического программирования. Ему принадлежит идея так называемых схем с переключениями по продуктивным и непродуктивным шагам \cite{Polyak_1967}. Общая идея подхода заключается в следующем: если в текущей точке значение ограничения достаточно хорошее, то спуск выполняем по целевой функции, а в противном случае~--- по функции ограничения. Такого типа подходам, которые интересны ввиду малых затрат памяти на итерациях, посвящаются всё новые работы как для выпуклых задач большой и сверхбольшой размерности \cite{Huang_2023, Bayandina_2018, Lagae_2017, nesterov2014subgradient}, так и для некоторых классов невыпуклых задач \cite{Huang_2023}. В последние годы некоторыми из авторов статьи были исследованы адаптивные субградиентные методы с переключениями для липшицевых задач выпуклого программирования, в том числе и для некоторых нелипшицевых и/или квазивыпyклых целевых функций \cite{Bayandina_2018, Stonyakin2018, Stonyakin2020, Stonyakin_timm2023}. Достаточно полное исследование стохастического варианта сyбградиентного метода с переключениями по продуктивным и непродуктивным шагам имеется в работе \cite{tiapkin2022primal}.

В этом пункте приводится результат о сходимости субградиентных методов со скоростью геометрической прогрессии для субградиентных схем с переключениями по продуктивным и непродуктивным шагам в случае выпуклых задач с ограничениями в виде неравенств. Случай квазивыпуклых задач с квазивыпуклыми ограничениями исследован в \cite{Stonyakin_timm2023}. Будем рассматривать  задачу с функциональными ограничениями вида
\begin{equation}\label{3}
\begin{cases}
\min f(x),\\
x\in Q, \; g(x) \leq 0,
\end{cases}
\end{equation}
где $f(x)$ и $g(x)$~--- липшицевы функции. Всюдy далее будем считать, что задача \eqref{3} разрешима.

Рассмотрим теперь алгоритм \ref{a7}, где $I$ и $J$ --- множества индексов продуктивных и непродуктивных шагов соответственно, а $|I|$ и $|J|$ --- мощности этих множеств.

\begin{algorithm}[htb]
\caption{Адаптивный субградиентный метод для выпуклой целевой функции.}
\label{a7}
\begin{algorithmic}[1]
\REQUIRE $\delta>0, \moh{M_g>0},  x^0, \theta_0: \theta_0^2 \geq \frac{1}{2} \|x_* - x^0\|_2^2,$ \moh{множество $\moh{Q}$.}  
\STATE $I=:\emptyset$
\STATE $N \leftarrow 0$
\REPEAT
\IF{$g\left( x^{\moh{N}} \right)\leq  \delta M_g$}
\STATE $h_{\moh{N}}^f = \frac{\delta}{\| \nabla f(x^{\moh{N}}) \|_2^2},$
\STATE $x^{\moh{N}+1} = \moh{\operatorname{Pr}_Q} \left(x^{\moh{N}} - h_{\moh{N}}^f \moh{\nabla} f(x^{\moh{N}})\right),$ \quad // \moh{\text{\emph{<<продуктивные шаги>>}}}
\STATE  $N \to I,$
\ELSE
%\STATE $g({\moh{N}}_k) >  M_g \delta$
\STATE $h_{\moh{N}}^g = \frac{\delta}{\| \nabla g(x^{\moh{N}})\|_2},$
\STATE $x^{\moh{N}+1} = \moh{\operatorname{Pr}_Q} \left(x^{\moh{N}} - h_{\moh{N}}^g  \nabla\moh{g}(x^\moh{N})\right),$ \quad // \moh{\text{\emph{<<непродуктивные шаги>>}}}
\ENDIF
\STATE $N\leftarrow N+1,$
\UNTIL {$\frac{2 \theta_0^2}{\delta^2} \leq \sum\limits_{k \in I} \frac{1}{\|\nabla f(x^k) \|_2^2} + N - \moh{|I|}.$}
\ENSURE $\moh{\widehat{x}:=\argmin_{x^k,\;k\in I}\,f(x^k)}.$ %$\moh{\widehat{x}} := \frac{1}{\sum\limits_{k\in I} h_k} \sum\limits_{k\in I} h_k x_k.$
\end{algorithmic}
\end{algorithm}

Покажем пару примеров, несколько поясняющих смысл использования таких схем для задач с ограничениями.

\begin{example}
Рассмотрим случай задачи с несколькими ограничениями
\begin{equation}
    g(x) = \max\{g_1(x) , ..., g_m(x)\}. 
\end{equation}
Можно сэкономить время работы алгоритма за счет рассмотрения не всех функциональных ограничений на непродуктивных шагах. То есть на непродyктивных шагах вместо субградиента ограничения $g(x)$ можно рассматривать субградиент любого из функционалов $g_m (x)$, для которого верно $g_{m}(x^k) > \varepsilon$. Другие ограничения при этом можно игнорировать. Легко проверить, что результат о сходимости алгоритма \ref{a7} при этом сохранится. Аналогичное замечание можно сделать и про иные подходы указанного типа \cite{Stonyakin2018}.
\end{example}

\begin{example} 
Интересным примером приложений может быть использование схем с переключениями к задачам проектирования механических констрyкций, сводящихся к минимизации функций max-типа с разреженной матрицей $A$ \cite{nesterov2014subgradient, Vorontsova_2021}

\begin{equation}\label{p5_4.7}
\max\limits_{y }\,\langle f,y\rangle:\quad g(y) := \max\limits_{1\leqslant i\leqslant d}\left( \pm\langle a_{i},y\rangle - 1 \right) = \max_{1 \leq i \leq 2d} g_i(y) \leqslant 0.
\end{equation}

Основное преимущество применения субградиентного метода с переключениями для задач выпуклого программирования заключается в том, что сложность одной итерации сильно понижается за счёт разреженности векторов $a_i,f$. Из этого следует, что на каждом шаге
$$y^{k+1} = y^k - h_g \cdot \nabla g(y^k) \qquad \mbox{или} \qquad y^{k+1} = y^k + h_f \cdot f$$
в векторе $y$ обновляется не больше чем $r = O(1)$ элементов. Из того, что в каждом узле встречаются не более чем $s$ стержней, следует, что обновление $y$ влечёт за собой обновление максимум $sr$ скалярных произведений $\langle a_i,y \rangle$.
\end{example}

%%% > %%%

\begin{theorem}
После остановки алгоритма~\ref{a7} для всякого $M_g$-липшицева квазивыпуклого ограничения $g$ верно $f(\widehat{x}) - f(x_*) \leq \delta$ и $g(\widehat{x}) \leq \delta M_g$. При этом в случае $M_f$-липшицевой функции $f$ достаточное количество итераций для выполнения критерия остановки алгоритма~\ref{a7} оценивается следующим образом:
\begin{equation*}
N \geq \frac{2 \theta_0^2 \max\{1, M_f^2\}}{\delta^2}.
\end{equation*}
\end{theorem}

\begin{remark}
Заметим, что на практике для реализации алгоритма \eqref{a7} знание константы Липшица $M_f$ функции $f(x)$ необязательно, достаточно остановить алгоритм после выполнения критерия остановки. 
\end{remark}

\subsection{Аналог условия острого минимума для задач с ограничениями}
Для выпуклых (в том числе негладких) липшицевых задач можно доказать сходимость субградиентного метода со скоростью геометрической прогрессии при дополнительном предположении о выполнении условия острого минимума \cite{Stonyakin_timm2023}. В частности \cite{поляк1983введение}, острым минимумом будет обладать задача вида
\begin{equation}\label{SharpMinTask}
\begin{cases}
\min \langle c, x \rangle; & x \in \mathbb{R}^n,\\
A x \leq b; & b \in \mathbb{R}^m,
\end{cases}
\end{equation}
где $c$~--- $n$-мерный вектор, $A$~--- матрица порядка $m \times n$. 

Если для линейных задач \eqref{SharpMinTask} возможно обойтись использованием обычного условия острого минимума \eqref{tochka}, то в общем случае для нелинейных задач это уже не так. 

Для такой постановки \eqref{3} будем использовать следующую вариацию понятия острого минимума \cite{Lin, Stonyakin_timm2023}.

\begin{definition}
Будем говорить, что для задачи вида \eqref{3} выполняется условие острого минимума (<<условный>> острый минимум), если при некотором $\alpha >0$ для всех $x$ справедливо неравенство
\begin{equation}\label{sm2}
\max\{f(x) - f^*, g(x) \} \geq \alpha \min_{x_* \in X_*} \|x-x_*\|_2.
\end{equation}
\end{definition}
Смысл такого подхода следующий. Поскольку задача \eqref{3} с ограничениями в виде неравенств, то в тех точках $x$, где эти неравенства нарушены, возможно $f(x) \leq f^*$, и тогда выполнение неравенства \eqref{sm2} возможно за счёт положительного $g$ (очевидный пример~--- когда в качестве ограничения выбирается функция расстояния от точки до допустимого множества). В ситуации же $g(x) \leq 0$ потенциально возможна ситуация, когда неравенство $f(x) - f^* \geq  \alpha \min\limits_{x_* \in X_*} \|x-x_*\|_2$ справедливо, а на всём допустимом множестве было бы неверным. Рассмотрим некоторые примеры таких задач с <<условным>> острым минимумом.

\begin{example}
Специальный вариант линейной задачи \cite{Lin}
\begin{equation}\label{problem_SLP}
\min -x_1,
\end{equation} 
\begin{equation}\label{constarints_SLP}
    \rho \cos \left ( \frac{j \pi}{10}\right ) x_1 + \rho \, \moh{\sin}\left ( \frac{j \pi}{10}\right ) x_2 \leq \rho, \quad j = 0, 1, \ldots, 19.
\end{equation}
Как известно \cite{Lin}, в этой задаче точка минимума $x_* = (1,0)$, а оптимальное значение $f^* = -1$. Отдельно целевая функция $-x_1$, вообще говоря, не удовлетворяет условию острого минимума \eqref{tochka}, поскольку зависит только от одной переменной из двух. Но при наличии ограничений \eqref{constarints_SLP} для этой задачи выполнен <<условный>> вариант острого минимума \eqref{sm2} при $\alpha = \frac{\rho}{2}$ \cite{Lin}. Выбор масштабирующего коэффициента $\rho$ может влиять на значение параметра такого острого минимума.
\end{example}

\begin{example}
Рассмотрим задачу
\begin{equation*}
\begin{cases}
\min \left \{f(x_1, x_2) = |x_1| + |x_2| \right \}, \\
g(x_1, x_2) = \max \{\varepsilon - x_1, \varepsilon - x_2 \} \leq 0,
\end{cases}
\end{equation*}
где $\varepsilon > 0.$ В этом примере целевая функция удовлетворяет условию острого минимума \eqref{tochka}. Действительно, $f(x) = |x_1| + |x_2|$, $f^* = 0$, а $x_* = (0, 0)$, и неравенство $|x_1| + |x_2| \geq \alpha \|x\|_2$ выполняется, например, при $\alpha=1$. Однако функциональные ограничения смещают точку минимума $x_*=(\varepsilon, \varepsilon)$, и условие острого минимума нарушается. А условному острому минимуму эта задача удовлетворяет, учитывая, что $x_* = (\varepsilon, \varepsilon)$, а минимальное значение функции $f^* = 2\varepsilon$.
\end{example}

\begin{example}
Задача классификации с ограничениями \cite{Lin}. Рассмотрим множество, состоящее из $n$ пар: $\mathcal{D} = \{(a_i, b_i)\}_{i=1}^n$, где $a\in \mathbb{R}^p$~--- вектор признаков, а $b_i \in \{1, -1\}$~--- множество меток при $i=1, 2, \ldots, n$. Пусть $\mathcal{D}_M \in \mathbb{R}^p$, $\mathcal{D}_F \in \mathbb{R}^p$~--- два типа признаков. Мы хотим найти линейный классификатор $x\in \mathbb{R}^p$, который не только минимизирует функцию потерь, но и правильным образом обрабатывает каждый элемент из множеств $\mathcal{D}_F$ и $\mathcal{D}_M$. Указанная задача сводится к следующей постановке:
\begin{equation*}
\min\limits_{x} \frac{1}{n} \sum\limits_{i=1}^n \max_x \{ 0, 1 - b_i a_i^T x \},
\end{equation*}
\begin{equation*}
\frac{1}{n_F} \sum\limits_{a\in \mathcal{D}_F} \sigma(a^T x) \geq \frac{\kappa}{n_M} \sum\limits_{a\in \mathcal{D}_M} \sigma(a^T x) ,
\end{equation*}
\begin{equation*}
\frac{1}{n_M} \sum\limits_{a\in \mathcal{D}_M} \sigma(a^T x) \geq \frac{\kappa}{n_F} \sum\limits_{a\in \mathcal{D}_F} \sigma(a^T x) ,
\end{equation*}
где $\kappa \in (0,1]$~--- константа, $n_M$ и $n_F$~--- количество экземпляров из $\mathcal{D}_M$ и $\mathcal{D}_F$ соответственно и $\sigma(a^T x) := \max\{0, \min\{1, \{0.5+a^T x\}\} \in [0,1]$~--- вероятность присваивания метки $+1$ предсказанию $a$.
\end{example}

Построим схему рестартов алгоритма~\ref{a7} (алгоритм~\ref{a8}) в предположении, что верно условие \eqref{sm2}.
\begin{algorithm}[H]
\caption{Рестарты алгоритма~\ref{a7}.}
\label{a8}
\begin{algorithmic}[1]
\REQUIRE $\varepsilon>0, \moh{\alpha >0, M_g>0},  x^0, \theta_0: \theta_0^2 \geq \frac{1}{2} \|x_* - x^0\|_2^2,$ \moh{множество $\moh{Q}$.}  
\STATE Set $p = 1.$
\REPEAT
\STATE \text{$\widehat{x}^p$~--- результат работы алгоритма~\ref{a7} с параметрами $\delta_p$, $\theta_p$, $x^0$,} \moh{где}
\STATE $x^0 = \widehat{x}^p,$
\STATE $\theta_p = \frac{1}{\sqrt{2^p}} \theta_0,$
\STATE $\delta_p = \frac{\alpha \theta_p}{\sqrt{2} \max\{1, M_g\}}.$
\STATE Set $p = p + 1.$
\UNTIL {$p > \left \lceil 2 \log_2 \frac{\theta_0}{\varepsilon} \right \rceil.$}
\ENSURE $x^p.$
\end{algorithmic}
\end{algorithm}

Для алгоритмов~\ref{a7} и~\ref{a8} верна следующая

\begin{theorem}\label{th} \cite{Stonyakin_timm2023}
Пусть $f(x)$ и $g(x)$~--- липшицевы функции с константами $M_f$ и $M_g$ соответственно, удовлетворяющие условию \eqref{sm2}, и известна константа $\theta_0 > 0$ такая, что $2 \theta_0^2 \geq \|x_* - x^0 \|_2^2$. Тогда для алгоритма~\ref{a7} можно подобрать параметр $\delta > 0$ так, чтобы после
$$\left \lceil \frac{4}{\alpha^2} \max\left\{1, M_f^2\right\} \max\left \{1, M_g^2 \right\} \right \rceil$$
итераций было выполнено неравенство $\min\limits_{x_* \in X_*} \|\widehat{x} -x_*\|_2 \leq \frac{1}{\sqrt{2}} \theta_0$. После $p-1$ запусков алгоритма~\ref{a8} (рестартов алгоритма~\ref{a7}) имеем
$$\min\limits_{x_*\in X_*} \|\widehat{x}^{p-1}-x_*\|_2 \leq  \frac{1}{\sqrt{2^p}} \theta_0 .$$
Тогда для достижения $\varepsilon$-точного решения вида
$$\min\limits_{x_*\in X_*} \|\widehat{x}^{p-1}-x_*\|_2 \leq  \varepsilon $$
достаточное количество обращений к субградиенту $f$ или $g$ можно оценить как
\begin{equation*}
\left \lceil \frac{4}{\alpha^2} \max\left\{1, M_f^2\right\} \max\left \{1, M_g^2 \right\} \right \rceil \left \lceil 2\log_2 \frac{\theta_0}{\varepsilon}\right \rceil.
\end{equation*}
\end{theorem}

\subsection{Субградиентные методы для слабо выпуклых и относительно слабо выпуклых задач с острым минимумом}

%Недавно в \cite{Davis_2018_Subgradient} полyчен резyльтат о сходимости субградиентного метода с шагом Б.\,Т.\,Поляка на классе слабо выпуклых функций, обладающих острым минимумом.

Введённый Б.Т. Поляком класс выпyклых негладких оптимизационных задач с острым минимумом интересен тем, что на нём сyбградиентный метод имеет неплохие вычислительные гарантии (сходится со скоростью геометрической прогрессии). Однако представляется, что условие острого минимума достаточно существенно сужает рассматриваемый класс выпуклых задач. В этой связи интересно отметить наблюдение многих статей последних лет \cite{Davis_2018_Subgradient,Duchi_2019_Solving,Eldar_2014_Phase,LI_NONCONVEX_ROBUST} о том, что многие возникающие в приложениях слабо выпуклые задачи имеют острый минимум. Это касается разных типов задач нелинейной регрессии, восстановления фазы, восстановления матрицы (matrix recovery) и др. Класс слабо выпуклых оптимизационных задач интересен, ему посвящаются всё новые работы \cite{Huang_2023,Dudov_2021, Li_2020_Incremental, Davis_2020_Nonsmooth, Davis_2018_Stochastic, Davis_2018_Subgradient, Stonyakin_crm2023}, но без дополнительных предположений скоростные гарантии достаточно пессимистичны. Однако в случае острого минимума недавно доказаны результаты о сходимости субградиентного метода со скоростью геометрической прогрессии при условии достаточной близости начальной точки к точному решению \cite{Davis_2018_Subgradient}. При этом заметим, что в одном из этих результатов используется способ выбора шага, предложенный Б.Т. Поляком. Немного расскажем об этом направлении.

Как и ранее, рассматриваем задачи минимизации вида \eqref{probl}.
Напомним, что функция $f:\mathbb{R}^{n} \longrightarrow \mathbb{R}$ (аналогично и для функций $f: Q \longrightarrow \mathbb{R}$) называется $\mu$-слабо выпуклой $(\mu\geqslant 0)$, если функция $x \longrightarrow f(x) +\frac{\mu}{2}\|x\|_2^{2}$ выпукла. Для недифференцируемых $\mu$-слабо выпуклых функций $f$ под субдифференциалом $\partial f(x)$ в точке $x$ можно понимать (см. \cite{Davis_2018_Subgradient} и цитируемую там литературу) множество всех векторов $\nabla f(x) \in \mathbb{R}^{n}$, удовлетворяющих неравенству
\begin{equation}\label{f_1}
f(y)\geqslant f(x)+\langle \nabla f(x), y-x \rangle + o(\|y-x\|_2)\quad \text{при}\;\; y \rightarrow x.
\end{equation}
Известно \cite{Davis_2018_Subgradient}, что все векторы-субградиенты $\nabla f(x) \in \mathbb{R}^{n}$ из \eqref{f_1} автоматически удовлетворяют неравенству
\begin{equation}\label{f_2}
f(y)\geqslant f(x)+\langle \nabla f(x), y-x\rangle  - \frac{\mu}{2}\|y-x\|_2^{2}\quad \forall\;x,y\in \mathbb{R}^{n}, \;\nabla f(x) \in \partial f(x).
\end{equation}
Можно проверить, что слабо выпуклые функции локально липшицевы, и поэтому в качестве субградиентов можно использовать произвольный вектор из субдифференциала Кларка (см., например, \cite{Dudov_2021}). 

Известно немало примеров прикладных слабо выпуклых задач, которые не являются выпуклыми. Так, слабо выпуклыми будут задачи вида (см., например, \cite{Davis_2018_Subgradient}):
$$
\min_{x} f(x):=h(c(x)),
$$
где $h: \mathbb{R}^{m} \longrightarrow \mathbb{R}$ выпукло и $M$-липшицево, а $c: \mathbb{R}^{n} \longrightarrow \mathbb{R}^{m}$ является $C^{1}$-гладким отображением с $\beta$-липшицевой матрицей Якоби. Нетрудно проверить, что при указанных допущениях $f$~--- $M\beta$-слабо выпукла. В частности, в указанный класс задач входят задачи нелинейной регрессии с $h(x) = \|x\|_1$, где $\|\cdot\|_1$~--- $1$-норма. 

Интересно, что слабая выпyклость позволяет сyщественно расширить класс задач с yсловием острого минимyма, предложенного Б.Т. Поляком в \cite{поляк1969минимизация}. Приведём ещё пару достаточно популярных в современных приложениях примеров слабо выпуклых задач с острым минимумом.

\begin{example}[Задача восстановления фазы]
Восстановление фазы~--- распространенная вычислительная задача, имеющая приложения в различных областях, таких как визуализация, рентгеновская кристаллография и обработка речи
\cite{Davis_2020_Nonsmooth,Duchi_2019_Solving,Eldar_2014_Phase}.

Она сводится к задаче минимизации следующей функции:
\begin{equation}\label{obj_phase}
f(x) = \frac{1}{m} \sum_{i=1}^m\mbr{\left\langle a_i, x\right\rangle^2-b_i},
\end{equation}
где $a_i \in \mathbb{R}^n$ и $b_i \in \mathbb{R}$ заданы для каждого $i=1, \ldots, m$. Данная целевая фyнкция имеет вид $f(x):=h(c(x)) \rightarrow \min_{x}$, где $h\colon\R^m\to\R$~--- выпуклая и $M$-липшицева функция, а $c\colon\R^n\to\R^m$~--- это $C^1$-гладкое отображение с $\beta$-липшицевым  отображением Якоби. Как отмечено выше, такие фyнкции слабо выпуклы. В данном случае функция \eqref{obj_phase} $\mu$-слабо выпукла~\cite{Li_2020_Incremental} при
\begin{equation}
\mu = 2 \max _{1 \leq i \leq m}\nbr{a_i}_2^2. 
\end{equation}

Более того, при соответствующей модели шума в измерениях фyнкция имеет острый минимyм (\cite{Duchi_2019_Solving}, Proposition~3).
\end{example}

\begin{example}\cite{LI_NONCONVEX_ROBUST}
\begin{equation}
\minimize_{U\in\R^{n\times r}} \left\{ \xi(U) := \frac{1}{m}\|y - A(U U^T)\|_2^2\right\}.
\label{eq:ms factorization}
\end{equation}
Рассмотрим {\em робастную задачу восстановления матрицы небольшого ранга} \cite{LI_NONCONVEX_ROBUST}, в которой измерения искажены {\em всплесками}. В частности, мы предполагаем, что
\begin{equation}\label{eq:rms model}
y = A(X^{\ast}) + s^{\ast},
\end{equation}
где $s^{\ast}\in\R^m$~--- вектор всплесков-возмущений, у которого небольшая часть элементов имеет произвольную величину, а остальные элементы равны нулю. Кроме того, множество ненулевых элементов предполагается неизвестным. 

Хорошо известно, что $\ell_2$-функция потерь чувствительна к возмущениям, что приводит к недостаточной эффективности постановки~\eqref{eq:ms factorization} для восстановления базовой матрицы низкого ранга. Глобальные минимумы $\xi$ в \eqref{eq:ms factorization} отклоняются от базовой матрицы низкого ранга из-за всплесков, и большая доля всплесков приводит к большему возмущению. Напротив, функция потерь $\ell_1$ более устойчива к выбросам и широко используется для обнаружения выбросов
\cite{Li_Sun_Chi,Candes,Josz}. Это привело к идее использовать функцию потерь $\ell_1$ и факторизованное представление матричной переменной для решения надежной задачи восстановления малоранговой матрицы :
\begin{equation}\label{eq:rms factorization}
\minimize_{U\in\R^{n\times r}} \left\{ f(U):= \frac{1}{m}\|y - A(U U^T) \|_1 \right\}.
\end{equation}
Известно \cite{LI_NONCONVEX_ROBUST}, что эта целевая функция слабо выпуклая и обладает острым минимумом.
\end{example}

Отметим, что для слабо выпуклых задач характерны многие проблемы сходимости вычислительных процедур, свойственные общей невыпуклой ситуации. Например, нет возможности гарантировать сходимость градиентного метода даже к локальному минимуму.  
\begin{example}
Действительно \cite{nesterov2018lectures}, пусть
$$f(x)\equiv f(x^{(1)},x^{(2)})=\frac{1}{2} \br{x^{(1)}}^2+\frac{1}{4}\br{x^{(2)}}^4-\frac{1}{2}\br{x^{(2)}}^2.$$
Это слабо выпуклая задача, поскольку добавление к $f$ половины квадрата нормы $x$ приводит к выпуклой функции.
Градиент целевой функции имеет вид $$\nabla f(x)=\br{x^{(1)},\br{x^{(2)}}^3-x^{(2)}}^T,$$ откуда следует, что существуют только три точки, которые могут претендовать на локальный минимум:\\ $x_1^*=(0,0),\;x_2^*=(0,-1),\;x_3^*=(0,1).$
Вычисляя матрицу Гессе
$$\nabla^2 f(x)=
\begin{pmatrix}
  1 & 0 \\
  0 & 3\br{x^{(2)}}^2-1
\end{pmatrix},$$
заключаем, что $x_2^*$ и $x_3^*$ являются точками изолированного локального минимума, в то время как $x_1^*$ есть только \emph{стационарная точка} нашей функции. Действительно, $f(x_1^*)=0$ и $f(x_1^*+\varepsilon e_2)=\br{\frac{\varepsilon^4}{4}}-\br{\frac{\varepsilon^2}{2}}<0$ при достаточно малых $\varepsilon$. Рассмотрим траекторию градиентного метода, начинающуюся в точке $x_0=(1,0)$. Обратим внимание на то, что вторая координата этой точки $0$, поэтому вторая координата для $\nabla f(x_0)$ также $0$. Следовательно, вторая координата точки $x_1$ равна нулю и т. д. Таким образом, вся последовательность точек, образованная градиентным методом, будет иметь нулевую вторую координату, что означает сходимость этой последовательности к $x_1^*$.
\end{example}

\begin{theorem}\label{Thm1}{\rm\cite{Davis_2018_Subgradient}}
Пусть $f$ $\mu$-слабо выпукла и имеет $\alpha$-острый минимум ($\alpha,\;\mu>0$), а точка $x_0$: $\min\limits_{x_*\in X_*} \|x^0 - x_*\|_2 \leq \frac{\alpha \gamma}{\mu}$ для некоторого фиксированного $\gamma \in (0;1) $. Тогда для метода \eqref{p0_eq2}
с шагом \eqref{h} верно неравенство
\begin{equation}\label{equation1n}
\min_{x_*\in X_*} \|x^{k+1}-x_*\|_2^2 \leq \prod_{i=0}^{k} \left(1-\frac{\alpha^2(1-\gamma)}{ \|\nabla f(x^i)\|_2^2}\right) \min_{x_* \in X_*} \|x^0-x_*\|_2^2.
\end{equation}
\end{theorem}

\begin{corollary}
Если в условиях предыдущей теоремы допустить, что $f$ удовлетворяет условию Липшица с константой $M >0$, то можно утверждать сходимость алгоритма~ \eqref{p0_eq2} со скоростью геометрической прогрессии
$$
\min_{x_* \in X_*}\|x^{k+1} - x_* \|_2^2 \leq \left( 1 - \frac{\alpha^2(1-\gamma)}{M^2} \right)^{k+1} \min_{x_* \in X_*}\|x^0 - x_* \|_2^2.
$$
\end{corollary}

Отметим, что в условиях теоремы \ref{Thm1} нулевой (суб)градиент возможен только в точке $x^k \in X_*$, поскольку справедливо следующее утверждение.

\begin{proposition}
[Окрестность множества $X_*$ не имеет стационарных точек]\label{lem:region_without_stat}
Задача минимизации $\mu$-слабо выпуклой функции $f$ c $\alpha$-острым минимумом не имеет стационарных точек $x$, удовлетворяющих условию
\begin{equation}\label{eq:region_for_linear}
0<\min_{x_* \in X_*}\|x - x_*\|_2< \frac{2\alpha}{\mu}.
\end{equation}
\end{proposition}
\begin{comment}
\begin{proof}
Зафиксируем стационарную точку $x\in Q \setminus X_*$ функции $f$. Выбрав произвольное $x_* \in\proj_{X_*}(x)$, заметим, что
\begin{align*}
\alpha\cdot \min_{x_* \in X_*}\|x - x_*\|_2 \leq f(x)-f(x_*) \leq \frac{\mu}{2}\|x - x_*\|_2^2=\frac{\mu}{2}\min_{x_* \in X_*}\|x - x_*\|_2^2.
\end{align*}
Разделив обе части последнего неравенства на $\min_{x_* \in X_*}\|x - x_*\|_2$, получим требуемое.
\end{proof}
\end{comment}

\section{Гладкая оптимизация}\label{sec:smooth}
Пожалуй, основным атрибутом гладкой оптимизации является наличие моментного ускорения, роль которого, по-видимому, впервые была отмечена в 1964 году в работе Б.Т.~Поляка \cite{поляк1964некоторых}. Отметим, что в этой работе (как и в работе по негладким методам 1969 года \cite{поляк1969минимизация}) были заложены основы разработки численных методов оптимизации на базе непрерывных аналогов (дифференциальных уравнений, порождающих при дискретизации итерационный процесс). Впоследствии аналогичным образом неоднократно разрабатывались новые численные методы оптимизации \cite{немировский1979сложность,евтушенко1982методы,su2014differential,wilson2021lyapunov}. Однако наличие гладкости дает возможность не только ускорения скорости сходимости. В данном разделе, следуя работам Б.Т.~Поляка, в основном 60-х годов прошлого века, продемонстрировано, какие дополнительные возможности приобретаются при наличии гладкости (липшицевости градиента целевой функции).     
\subsection{Условие градиентного доминирования (Поляка--Лоясиевича)}\label{sec:PL}

В 1963 году Б.Т. Поляк \cite{поляк1963градиентные} предложил простое условие, достаточное для того, чтобы показать глобальную линейную скорость сходимости градиентного спуска для достаточно гладких задач. Это условие является частным случаем неравенства Лоясиевича, предложенного в том же году \cite{lojasiewicz1963propriete}, и не требует сильной выпуклости (или даже выпуклости). Условие вида \eqref{eq:PL:PL} называется условием градиентного доминирования, или же условием Поляка--Лоясиевича (PL-условием) \cite{поляк1963градиентные, lojasiewicz1963propriete, lezanski1963minimumproblem, karimi2016linear}:
\begin{equation}
    \label{eq:PL:PL}
    f(x) - f(x_*) \leq \frac{1}{2 \mu} \lVert \nabla f(x) \rVert_2^2 \quad \forall x \in \mathbb{R}^n.
\end{equation}
Как правило, в литературе это условие называют условием Поляка--Лоясиевича, но стоит отметить ещё работу Лежанского \cite{lezanski1963minimumproblem}, в которой оно появилось независимо. По-видимому, правильнее называть \eqref{eq:PL:PL} условием Лежанского--Поляка--Лоясиевича. Зачастую в теоретических приложениях данное условие используется как ослабление сильной выпуклости, так как любая $\mu$-сильно выпуклая функция удовлетворяет PL-условию с константой $\mu$.

В последние годы условие градиентного доминирования обширно исследуется в различных областях оптимизации и смежных науках. Толчком к возрождению интереса к этому условию послужила работа \cite{karimi2016linear}, в которой авторы показали, что PL-условие слабее тех условий, которые ранее использовались для получения линейной скорости сходимости без сильной выпуклости. Также в этой работе PL-условие использовалось для проведения нового анализа рандомизированных и жадных методов покоординатного спуска, а также стохастических градиентных методов в различных постановках. Было предложено и обобщение результатов на проксимальные градиентные методы, позволяющее несложно доказать их линейную скорость сходимости. Попутно были получены результаты сходимости для широкого круга задач машинного обучения: метода наименьших квадратов, логистической регрессии, бустинга, устойчивого метода обратного распространения ошибки, L1-регуляризации, метода опорных векторов.

Приведем несколько примеров задач, в которых условие градиентного доминирования возникает естественным образом.

\begin{example}[PL-условие для нелинейных задач]
    Пусть $f(x) = \frac{1}{2} \lVert \Phi(x) - y \rVert_2^2$, где ${\Phi: \mathbb{R}^d \to \mathbb{R}^ n}$ дифференцируемо. Тогда $f$ удовлетворяет PL-условию, если существует такое $\mu > 0$, что для всех $x \in \mathbb{R}^d$ имеет место равномерная невырожденность матрицы Якоби:
    \begin{equation}
        \label{eq:PL:lec16eq3}
        \lambda^{\min}\br{\frac{\partial \Phi(x)}{\partial x}\cdot\sbr{\frac{\partial \Phi(x)}{\partial x}}^T}\geq\mu,
    \end{equation}
    где
    \begin{equation*}
        \frac{\partial \Phi(x)}{\partial x} = \left( \frac{\partial \Phi_i(x)}{\partial x_j} \right)_{i,j = 1}^{n,d} = D \Phi(x).
    \end{equation*}
    Действительно, достаточно написать
    \begin{equation*}
        \lVert \nabla f(x) \rVert_2^2
        = \lVert D\Phi(x)^T(\Phi(x) - y) \rVert_2^2
        \geq \mu \lVert \Phi(x) - y \rVert_2^2
        = 2 \mu f(x)
        \geq 2 \mu(f(x) - f(x_*)).
    \end{equation*}
    
    В качестве $\Phi(x)$ может выступать, например, система нелинейных уравнений. В частности, из этого можно получить, что функции Розенброка и Нестерова--Скокова локально удовлетворяют условию Поляка--Лоясиевича.

    Обратим внимание, что допущение \eqref{eq:PL:lec16eq3} предполагает $d \geq n$. Это верно, если, например, $\Phi$ задаёт перепараметризованную нейронную сеть. Более точные рассуждения с опорой на использование структуры нейронной сети см.~в \cite{liu2020toward}.
\end{example}

В некоторых задачах целевая функция обладает PL-условием не на всём пространстве, а только на каком-то множестве. Однако если траектория вычислительного метода не выводит итеративную последовательность за рамки этого множества, то, конечно, можно применять PL-условие для анализа сходимости метода.

\begin{example}[Композиция сильно (строго) выпуклой и линейной функций]
    В \cite{karimi2016linear} показано, что если функция $f$ имеет вид $f(x) = g(Ax)$, где $g$~--- сильно выпукла, то $f$ удовлетворяет PL-условию. Функции данного вида часто возникают в машинном обучении (например, метод наименьших квадратов). Если ослабить условие сильной выпуклости функции $g$ до строгой выпуклости, то функция $f$ также будет удовлетворять PL-условию, однако уже не на всем пространстве, а только на произвольном компакте. Например, из данного факта следует, что целевая функция задачи логистической регрессии
    \begin{equation*}
        f(x) = \sum\limits_{i = 1}^n \log\left( 1 + \exp(b_i a_i^T x) \right)    
    \end{equation*}
    локально (т.е. на всяком компакте) обладает PL-условием.
\end{example}

\textcolor{red}{Также оказывается, что PL-условие (или его различные варианты) оказывается полезным в таких областях как теория оптимального управления \cite{fatkhullin2021optimizing} и обучение с подкреплением \cite{xiao2022convergence, ding2022global}, толчком в развитии приведённых работ в которых послужила работа \cite{fazel2018global}.} \textcolor{red}{Приведём пример} функции с условием градиентного доминирования, связанной с теорией оптимального управления. Этот достаточно интересный класс функций с условием градиентного доминирования был получен совсем недавно \textcolor{red}{И.Ф.~Фатхуллиным} и Б.Т.~Поляком в \cite{fatkhullin2021optimizing}.

\begin{example}[\cite{fatkhullin2021optimizing}]
    \color{red}
    Пусть задана линейная система управления
    \begin{equation*}
        \Dot{x}(t) = Ax(t) + Bu(t),
    \end{equation*}
    где $x(t) \in \mathbb{R}^n$~--- состояние системы, а $u(t) \in \mathbb{R}^p$~--- управление; с начальными условиями $x(0)$, распределенными случайно с нулевым средним и ковариационной матрицей $\Sigma$ ($\mathbb{E} x(0)x(0)^T = \Sigma$); с критерием качества вида
    \begin{equation*}
        f(K) = \mathbb{E} \int\limits_0^{\infty} \left( x(t)^T Q x(t) + u(t)^T R u(t) \right) dt, \quad Q, R \succ 0,
    \end{equation*}
    заданным на множестве $S$, с целью найти матрицу обратной связи по состоянию $u(t) = - Kx(t)$, минимизирующую $f(K)$.
   
    Тогда если известен $K_0$~--- стабилизирующий регулятор, то градиент $\nabla f(K)$ удовлетворяет условию Липшица на множестве уровня $S_0 = \left\{ K: f(K) \leq f(K_0) \right\}$, а функция $f(K)$ удовлетворяет условию градиентного доминирования со следующей константой $\mu$:
    \begin{equation*}
        \mu = \frac{
            \lambda_1(R) \lambda_1^2(\Sigma) \lambda_1(Q)
        }{
            8 f(K_*) \left( \lVert A \rVert_2 + \frac{
                \lVert B \rVert_2^2 f(K_0)
            }{
                \lambda_1(\Sigma) \lambda_1(R)
            } \right)^2
        },
    \end{equation*}
    где $K_*$~--- точка минимума $f(K)$ на множестве $S$. Здесь $\lambda_i (X)$~--- собственные числа произвольной квадратной матрицы $X \in \mathbb{R}^{m \times m}$, занумерованные в порядке возрастания их действительных частей, т.е. $Re \lambda_1(X) \leq Re \lambda_2(X) \leq \ldots \leq Re \lambda_m(X)$.
    
    Отметим, что функция $f(K)$ достаточно гладкая (имеет липшицев градиент), но невыпуклая.
\end{example}

Таким образом, условие Поляка--Лоясиевича верно для достаточно большого класса невыпyклых задач. Тем не менее, оно позволяет обосновать сходимость градиентного метода со скоростью геометрической прогрессии. Приведём оценку скорости сходимости градиентного метода для $L$-гладких функций с условием Поляка--Лоясиевича. Будем рассматривать задачи минимизации функции $f$, удовлетворяющей PL-условию \eqref{eq:PL:PL}, а также условию Липшица градиента
\begin{equation}
    \label{eq:PL:Lcond}
    \|\nabla f(x)-\nabla f(y)\|_2 \leq L\|x-y\|_2 \quad \forall\,x,\,y\in\mathbb{R}^{n}.
\end{equation}
Справедлива следующая
\begin{theorem}\label{th:PL:LPL}
    Пусть дана задача безусловной оптимизации $\argmin f(x)$, где функция $f$ является L-гладкой и удовлетворяет PL-условию с константой $\mu$. Тогда градиентный метод с постоянным шагом $\frac{1}{L}$ вида
    \begin{equation}
        \label{eq:PL:gd}
        x^{k + 1} = x^k - \frac{1}{L} \nabla f(x^k)
    \end{equation}
    имеет линейную скорость сходимости, то есть
    \begin{equation}
        \label{eq:PL:LPL}
        f(x^k) - f(x_*) \leq \left( 1 - \frac{\mu}{L} \right)^k (f(x^0) - f(x_*)) \leq \exp\left(-\frac{\mu}{L} k\right) (f(x^0) - f(x_*)).
    \end{equation}
\end{theorem}
\begin{proof}
    В силу $L$-гладкости функции $f$
    \begin{equation*}
        f(x^{k + 1}) \leq f(x^k) + \langle \nabla f(x^k), x^{k + 1} - x^k \rangle + \frac{L}{2} \lVert x^{k + 1} - x^k \rVert_2^2.
    \end{equation*}
    Пользуясь правилом обновления \eqref{eq:PL:gd}, получаем
    \begin{equation}
        \label{eq:PL:LPL1}
        f(x^{k + 1}) \leq f(x^k) - \frac{1}{2L} \lVert \nabla f(x^k) \rVert_2^2.
    \end{equation}
    Из PL-условия имеем
    \begin{equation}
        \label{eq:PL:LPL2}
        f(x^k) - f(x_*) \leq \frac{1}{2\mu} \lVert \nabla f(x^k) \rVert_2^2.
    \end{equation}
    Из \eqref{eq:PL:LPL1} и \eqref{eq:PL:LPL2} путем преобразований следует
    \begin{equation*}
        f(x^{k + 1}) - f(x_*) \leq \left( 1 - \frac{\mu}{L} \right) (f(x^k) - f(x_*)).
    \end{equation*}
    Рекурсивное применение последнего неравенства и дает требуемый результат \eqref{eq:PL:LPL}.
\end{proof}

Доказанная верхняя оценка скорости сходимости из теоремы~\ref{th:PL:LPL} известна уже около 60 лет. До недавнего времени вопрос об её оптимальности оставался открытым. В последние дни 2022 года был выложен препринт \cite{yue2023lower}, в котором обоснована оптимальность этой оценки на классе гладких задач с условием Поляка--Лоясиевича.

% \begin{remark}[\cite{гасников2018современные, поляк1983введение}]
%     В условиях теоремы \ref{th:PL:LPL} также выполняются соотношения
%     \begin{equation}\label{eq-norm-nonconv}
%         \min_{k=0,\ldots,N-1}\|\nabla f(x^k)\|_2\leq\sqrt{\frac{2L(f(x^0)-f(x_*))}{N}}
%     \end{equation}
%     и
%     \begin{equation}
%         \label{eq:PL:traj}
%         \|x_*-x^0\|_2\leq \frac{\sqrt{2L(f(x^0)-f(x_*))}}{\mu},
%     \end{equation}
% где $x_*$ --- ближайшее к $x^0$ точное решение задачи минимизации $f$. Отметим, что комбинирование справедливой для произвольной невыпyклой задачи оценки \eqref{eq-norm-nonconv} и условия градиентного доминирования позволяет ввести рестартованный вариант градиентного метода и обосновать его сходимость со скоростью геометрической прогрессии. Однако, в отличие от теоремы \ref{th:PL:LPL}, для практической реализации такой схемы необходимо знать параметр $\mu$.
% \end{remark}

Отметим, что существует аналог условия Поляка--Лоясиевича для седловых задач, так называемое двухстороннее условие Поляка--Лоясиевича, которое позволяет получить ряд схожих результатов \cite{yang2020global, garg2022fixed, nouiehed2019solving}. Будем рассматривать седловую задачу $\min\limits_x \max\limits_y f(x, y)$. Говорят, что функция $f(x, y)$ удовлетворяет двустороннему PL-условию, если существуют такие константы $\mu_1$ и $\mu_2$, что $\forall x,y$:
\begin{equation*}
    \begin{cases}
        \lVert \nabla_x f(x, y) \rVert_2^2 \geq 2\mu_1 \left( f(x, y) - \min\limits_x f(x, y) \right);\\
        \lVert \nabla_y f(x, y) \rVert_2^2 \geq 2\mu_2 \left( \max\limits_y f(x, y) - f(x, y) \right).
    \end{cases}
\end{equation*}
Примеры функций, удовлетворяющих двустороннему PL-условию:
\begin{itemize}
    \item $f(x, y) = F(Ax, By)$, где $F(\cdot,\cdot)$~--- сильно-выпукло-сильно-вогнутая и $A, B$~--- произвольные матрицы.
    \item Невыпукло-невогнутая $f(x, y) = x^2 + 3\sin^2 x \sin^2 y - 4y^2 - 10\sin^2y$ ($\mu_1 = 1/16$, $\mu_2 = 1/14$) \cite{yang2020global}.
    \item Релаксация задачи Robust Least Squares (RLS):
        \begin{gather*}
            \min\limits_{x \in \mathbb{R}^n} \max\limits_{y \in \mathbb{R}^m} F(x, y);\\
            F(x, y) \coloneqq \lVert Ax - y \rVert_M^2 - \lambda \lVert y - y_0 \rVert_M^2,
        \end{gather*}
        где $M$~--- положительно полуопределена и $\lVert x \rVert_M^2 = x^T M x$. RLS минимизирует невязку наихудшего случая при заданном ограниченном детерминированном возмущении $\delta$ на зашумленном векторе измерений $y_0 \in \mathbb{R}^m$ и матрице коэффициентов $A \in \mathbb{R}^{m \times n}$. Саму задачу RLS можно сформулировать так ($\rho > 0$) \cite{el1997robust}:
        \begin{equation*}
            \min\limits_{x \in \mathbb{R}^n} \max\limits_{\delta: \lVert \delta \rVert_2 \leq \rho} \lVert Ax - y \rVert_2^2,\text{ где }\delta = y_0 - y \in \mathbb{R}^m.
        \end{equation*}
        Известно, что при $\lambda > 1$ $F(x, y)$ удовлетворяет двустороннему PL-условию, поскольку $F(x, y)$ можно представить как комбинацию аффинной функции и сильно-выпукло-сильно-вогнутой функции.
\end{itemize}

Для седловых задач с двусторонним условием градиентного доминирования может использоваться вариация градиентного метода --- метод градиентного спуска-подъема. Этот подход сейчас довольно популярен \cite{yang2020global, муратиди2023правила, garg2022fixed}, но был подробно проанализирован ещё в малоизвестной статье Бакушинского--Поляка \cite{бакушинский1974решении}.

%\subsection{Влияние неточности информации о градиенте на качество выдаваемого градиентным методом решения. Результаты о сходимости и возможные правила ранней остановки для оптимизационных задач с yсловием градиентного доминирования}

%Для описанного в предыдyщем пyнкте класса достаточно гладких задач с yсловием градиентного доминирования поговорим о поведении градиентного метода градиентного метода, когда градиент целевой функции доступен с некоторой погрешностью.
Далее, поговорим о поведении градиентного метода (сходимость и возможные правила ранней остановки) для класса достаточно гладких задач с условием градиентного доминирования, когда градиент целевой функции доступен с некоторой погрешностью. Б.Т. Поляк считал важными вопросы исследования влияния погрешностей доступной информации на гарантии сходимости численных методов, занимался ими и уделил им немало внимания в своей книге  \cite{поляк1983введение}. Так, этой тематике посвящена одна из его последних работ \cite{stonyakin2023stopping}, о которой мы немного расскажем.
%Борис Теодорович Поляк придавал считал важными вопросы исследования влияния на гарантии сходимости численных методов погрешностей достyпной информации, занимался ими и yделил им немало внимание в своей книге  \cite{поляк1983введение}.

Сфокусируемся на двух основных известных типах неточной информации о градиенте: абсолютная и относительная погрешности. Типичными примерами областей, в которых возникает данная проблема неточного градиента, являются безградиентная оптимизация, в которой используется приближенно вычисляемый градиент \cite{berahas2022theoretical, conn2009introduction, risteski2016algorithms}, а также оптимизация в бесконечномерных пространствах, связанная с обратными задачами \cite{gasnikov2017convex, kabanikhin2011inverse}.

Начнем с вопроса влияния на качество выдаваемого решения абсолютной погрешности в градиенте. Будем рассматривать задачу минимизации функции $f$, удовлетворяющей PL-условию \eqref{eq:PL:PL}, а также условию Липшица градиента \eqref{eq:PL:Lcond}. Считаем, что методу доступно не точное, а приближённое значение градиента $\widetilde{\nabla}f(x)$ в любой запрашиваемой точке $x$:

\begin{equation*}
    \nabla f(x)=\widetilde{\nabla}f(x)+ v(x),\quad\text{причём}\quad \|v(x)\|_2\leq\Delta
\end{equation*}
при некотором фиксированном $\Delta>0$. Тогда \eqref{eq:PL:PL} означает, что
\begin{equation}
    \label{eq:PL:f3}
    f(x)-f(x_*)\leq\frac{1}{\mu}(\|\widetilde{\nabla}f(x)\|_2^2+\Delta^{2})\quad \forall\,x\in\mathbb{R}^{n},
\end{equation}
% поэтому для всякого $x$ верно $\|\widetilde{\nabla}f(x)\|_2^2+\Delta^{2}\geq\mu(f(x)-f(x_*))$, откуда
откуда
\begin{equation*}
    \|\widetilde{\nabla}f(x)\|_2^2\geq\mu(f(x)-f(x_*))-\Delta^{2}\quad \forall\,x\in\mathbb{R}^{n}.
\end{equation*}

Заметим, что вопросы исследования влияния погрешностей градиента на оценки скорости сходимости методов первого порядка ужа достаточно давно привлекают многих исследователей (см., например, \cite{поляк1983введение, devolder2014first, devolder2013exactness, d2008smooth, vasin2021stopping}). Однако мы сфокусируемся именно на выделенном классе, вообще говоря, невыпуклых задач. Невыпуклость целевой функции задачи, а также использование на итерациях неточно заданного градиента могут приводить к разным проблемам. В частности, при отсутствии каких-либо правил ранней остановки возможно достаточно большое удаление траектории градиентного метода от начальной точки. Это может быть проблемным, если начальное положение метода уже обладает определёнными хорошими свойствами. С другой стороны, неограниченное удаление траектории градиентного метода в случае градиентного метода с помехами может привести к удалению от искомого точного решения. Опишем некоторые ситуации такого типа.

В качестве простого примера не сильно выпуклой функции, удовлетворяющей условию градиентного доминирования, можно рассмотреть
\begin{equation}
    \label{eq:PL:simple_ex_1}
    f(x) = \langle Ax, x \rangle,  
\end{equation}
где $A = \text{diag}(L, \mu, 0)$~--- диагональная матрица $3$-го порядка с ровно двумя положительными элементами $L > \mu > 0$. Если для задачи минимизации функции \eqref{eq:PL:simple_ex_1} допустить наличие погрешности градиента $v(x) = (0, 0, \Delta)$ при $\Delta>0$, то при $x^0 = (0, 0, 0)$ и $h_k > 0$ последовательность $x^{k+1} = x^k - h_k \widetilde{\nabla}f(x^k)$ стремится к бесконечности при неограниченном увеличении $k$. Далее, можно рассмотреть функцию Розенброка двух переменных $x = (x_{(1)}, x_{(2)})$
$$
f(x) = 100\left(x_{(2)} - \left(x_{(1)}\right)^2\right)^2 + \left(1 - x_{(1)}\right)^2.
$$
При $x^0 = (1, 1) = x_*$ на каждом шаге градиентного метода возможна такая погрешность градиента $v(x^k)$, что $x^k_{(2)} = \left(x^k_{(1)}\right)^2$ и без остановки траектория может уходить весьма далеко от точного решения $x_*$. Аналогично траектория градиентного метода может уходить к бесконечности для целевой функции двух переменных $f(x) = (x_{(2)} - (x_{(1)})^2)^2$.

Немного расскажем об одной из последних работ, подготовленных Б.Т. Поляком с соавторами \cite{stonyakin2023stopping}. В этой статье поставлена цель исследовать оценку расстояния $\|x^N - x^0\|_2$ для выдаваемых градиентным методом точек $x^N$ и предложить правило ранней остановки, которое гарантирует некоторый компромисс между стремлением достичь приемлемого качества выдаваемого методом решения задачи по функции и не столь существенным удалением траектории от выбранной начальной точки метода. Отметим, что правила ранней остановки в итеративных процедурах активно исследуются для самых разных типов задач. По-видимому, впервые идеология раннего прерывания итераций была предложена в статье \cite{емелин1978правило}, посвящённой методике приближённого решения возникающих при регуляризации некорректных или плохо обусловленных задач (в упомянутой работе рассматривалась задача решения линейного уравнения). Ранняя остановка в этом случае нацелена на преодоление проблемы потенциального накопления погрешностей регуляризации исходной задачи. К данной тематике примыкают известные подходы, связанные с ранней остановкой методов первого порядка в случае использования на итерациях неточной информации о градиенте (см. \cite{поляк1983введение}, гл. 6, параграф 1, а также, к примеру, недавний препринт \cite{vasin2021stopping}). Однако известные прежде результаты для выпуклых (не сильно выпуклых) задач отличаются по сравнению с полученным в рассматриваемой работе тем, что обычно гарантируется достижение худшего уровня по функции (если сравнить с комментарием после теоремы 2 параграфа 1 гл. 6 из \cite{поляк1983введение}) или оценки вида $\|x^N - x_*\|_2 \leq \|x^0 - x_*\|_2$ без исследования $\|x^N - x^0\|_2$, где $\{x^k\}_{k \in \mathbb{N}}$ --- образуемая методом последовательность, $x_*$~--- ближайшее к точке старта метода $x^0$ точное решение задачи минимизации $f$.

В предположении доступности значений параметров $L>0$ и $\Delta>0$ применим к задаче минимизации $f$ градиентный метод вида
\begin{equation}
    \label{eq:PL:f5}
    x^{k+1}=x^k-\frac{1}{L}\widetilde{\nabla}f(x^k).
\end{equation}
Тогда ввиду~\eqref{eq:PL:Lcond} для метода \eqref{eq:PL:f5} имеем следующие соотношения:
$$
f(x^{k+1})\leq f(x^k)+\langle \nabla f(x^k), x^{k+1}-x^k\rangle+\frac{L}{2}\|x^{k+1}-x^k\|_2^2 =
$$
$$
= f(x^k) + \frac{1}{2L} \|\nabla f(x^k) - \widetilde{\nabla} f(x^k)\|_2^2 - \frac{\|\nabla f(x^k)\|_2^2}{2L} \leq
$$
$$
\leq f(x^k) + \frac{\Delta^2}{2L} -\frac{1}{2L}\|\nabla f(x^k)\|_2^2,
$$
т.е.
\begin{equation}
    \label{eq:PL:f6}
    f(x^{k+1}) - f(x^k)\leq \frac{\Delta^2}{2L} - \frac{1}{2L}\|\nabla f(x^k)\|_2^2.
\end{equation}
Суммирование неравенств \eqref{eq:PL:f6} приводит нас к оценке 
\begin{equation}
    \label{eq:PL:eqnormgrad}
    \min\limits_{k = 0, N-1}\|\nabla f(x^k)\|_2 \leq \sqrt{\Delta^2 + \frac{2L(f(x^0)-f(x_*))}{N}} \leq \Delta +  \sqrt{\frac{2L(f(x^0)-f(x_*))}{N}},
\end{equation}
что указывает на потенциальную возможность расходимости градиентного метода с абсолютной погрешностью задания градиента. Конкретные примеры таких ситуаций описаны выше.

С учётом \eqref{eq:PL:PL} тогда получаем, что
$$
f(x^{k+1}) - f(x^k)\leq \frac{\Delta^2}{2L}-\frac{2\mu(f(x^k)-f(x_*))}{2L} = - \frac{\mu}{L}(f(x^k)-f(x_*))+\frac{\Delta^2}{2L},
$$
откуда
$$
f(x^{k+1})-f(x_*)\leq \left(1- \frac{\mu}{L}\right)(f(x^k)-f(x_*))+\frac{\Delta^2}{2L}\leq
$$
$$
\leq\left(1- \frac{\mu}{L}\right)^{k+1}(f(x^0)-f(x_*))+\frac{\Delta^2}{2L}\left(1+1-\frac{\mu}{L}+\cdots +\left(1-\frac{\mu}{L}\right)^k\right)<
$$
$$
<\left(1- \frac{\mu}{L}\right)^{k+1}(f(x^0)-f(x_*))+\frac{\Delta^2}{2\mu},
$$
т.е.
\begin{equation}
    \label{eq:PL:f7}
    f(x^{k+1})-f(x_*)\leq \left(1- \frac{\mu}{L}\right)^{k+1}(f(x^0)-f(x_*))+\frac{\Delta^2}{2\mu}.
\end{equation}

Оценки \eqref{eq:PL:eqnormgrad} и \eqref{eq:PL:f7}, вообще говоря, неулучшаемы. К примеру, известны нижние оценки уровня точности по функции  $O\left(\frac{\Delta^2}{2\mu}\right)$ для градиентного метода с абсолютной погрешностью задания градиента (см., например, раздел 2.11.1 пособия \cite{Vorontsova_2021}, а также имеющиеся там ссылки) даже на классе сильно выпуклых функций.

В \cite{stonyakin2023stopping} для градиентного метода с постоянным шагом при достаточно малом значении возмущённого градиента получена теорема, в которой указан уровень точности по функции, который возможно гарантировать после выполнения предложенного правила ранней остановки. Данный результат можно применять ко всяким $L$-гладким невыпуклым задачам. Далее, с использованием PL-условия получено уточнение этого результата~--- достаточное количество итераций градиентного метода для достижения желаемого качества точки выхода $\widehat{x}$ по функции $f(\widehat{x}) - f(x_*) = O\left(\frac{\Delta^2}{\mu}\right)$.

Теперь перейдем к ситуации относительной неточности градиента. Будем рассматривать  поведения методов градиентного типа для выделенного класса достаточно гладких задач с условием Поляка--Лоясиевича в предположении доступности для метода в каждой текущей точке градиента с относительной погрешностью, т.е.
\begin{equation}\label{eq:PL:rel_err}
    \lVert \Tilde{\nabla} f(x) - \nabla f(x) \rVert_2 \leq \alpha \lVert \nabla f(x) \rVert_2,
\end{equation}
где под $\Tilde{\nabla} f(x)$, как и ранее, мы понимаем неточный градиент, а $\alpha \in [0, 1)$~--- некоторая константа, указывающая на величину относительной погрешности градиента. Данное условие на неточный градиент было введено и изучено в работах \cite{поляк1983введение, carter1991global}.

Из \eqref{eq:PL:rel_err} можно получить вариант условия градиентного доминирования \eqref{eq:PL:PL} для относительно неточного градиента
\begin{equation*}
    f(x) - f(x_*) \leq \frac{1}{2 \mu (1 - \alpha)^2} \lVert \Tilde{\nabla} f(x) \rVert_2^2.
\end{equation*}
Описанную задачу предлагается так же решать методом градиентного типа вида
\begin{equation}\label{eq:PL:method}
    x^{k + 1} = x^k - h_k \Tilde{\nabla} f(x^k).
\end{equation}
Оказывается, что можно подобрать так параметры шагов $h_k$, чтобы гарантировать сохранение результата о сходимости метода со скоростью геометрической прогрессии в случае относительной неточности градиента (хоть и с большим знаменателем). 

Если при реализации метода \eqref{eq:PL:method} градиент доступен с известной относительной погрешностью ${\alpha \in [0, 1)}$ и известен параметр $L > 0$, то выбор шага $h_k = \frac{1}{L} \frac{(1 - \alpha)}{(1 + \alpha)^2}$ приводит к следующим результатам (см.~параграф~1 из пособия \cite{гасников2018современные} и имеющиеся там ссылки).
\begin{equation*}
    \begin{aligned}
        f(x^{k + 1}) - f(x^k)
        &\leq \langle \nabla f(x^k), x^{k + 1} - x^k \rangle + \frac{L}{2} \lVert x^{k + 1} - x^k \rVert_2^2 \\
        &= -h_k \langle \nabla f(x^k), \Tilde{\nabla} f(x^k) \rangle + \frac{L h_k^2}{2} \lVert \Tilde{\nabla} f(x^k) \rVert_2^2 \\
        &= -h_k \lVert \nabla f(x^k) \rVert_2^2 - h_k \langle \nabla f(x^k), \Tilde{\nabla} f(x^k) - \nabla f(x^k) \rangle + \frac{L h_k^2}{2} \lVert \Tilde{\nabla} f(x^k) \rVert_2^2 \\
        &\leq -h_k \lVert \nabla f(x^k) \rVert_2^2 + h_k \lVert \nabla f(x^k) \rVert_2 \lVert \Tilde{\nabla} f(x^k) - \nabla f(x^k) \rVert_2 + \frac{L h_k^2}{2} \lVert \Tilde{\nabla} f(x^k) \rVert_2^2 \\
        &\leq -h_k \lVert \nabla f(x^k) \rVert_2^2 + \alpha h_k \lVert \nabla f(x^k) \rVert_2^2 + (1 + \alpha)^2 \frac{L h_k^2}{2} \lVert \nabla f(x^k) \rVert_2^2 \\
        &= \left( -(1 - \alpha) h_k + (1 + \alpha)^2 \frac{L h_k^2}{2} \right) \lVert \nabla f(x^k) \rVert_2^2.
    \end{aligned}
\end{equation*}
После подстановки $h_k$ (которое, заметим, минимизирует выражение в скобках в последнем выражении) получаем
\begin{equation}\label{eq:PL:funcPLconst}
    f(x^{k + 1}) - f(x^k) \leq - \frac{1}{2L} \frac{(1 - \alpha)^2}{(1 + \alpha)^2} \left\lVert \nabla f(x^k) \right\rVert_2^2.
\end{equation}
Пользуясь условием Поляка--Лоясиевича \eqref{eq:PL:PL}, имеем следующую оценку на скорость сходимости по функции
\begin{equation*}
    f(x^{k + 1}) - f(x_*) \leq \left(1 - \frac{\mu}{L} \frac{(1 - \alpha)^2}{(1 + \alpha)^2} \right) \left( f(x^k) - f(x_*) \right),
\end{equation*}
т.е.
\begin{equation*}
    f(x^N) - f(x_*) \leq \left(1 - \frac{\mu}{L} \frac{(1 - \alpha)^2}{(1 + \alpha)^2} \right)^N \left( f(x^0) - f(x_*) \right).
\end{equation*}
Если вместо PL-условия \eqref{eq:PL:PL} предполагать выполнение условия сильной выпуклости, то приведенную оценку можно уточнить \cite{de2017worst}:
\begin{equation*}
    \frac{(1 - \alpha)^2}{(1 + \alpha)^2} \to O\left( \frac{1 - \alpha}{1 + \alpha} \right).
\end{equation*}

В работе \cite{puchinin2023gradient} предлагается адаптивный алгоритм для случая относительной погрешности градиента. По сути, предлагается критерий выхода из итерации, содержащий норму неточного градиента $\lVert \Tilde{\nabla} f(x^k) \rVert_2$. Это обстоятельство затрудняет использование подхода с оценками типа \eqref{eq:PL:funcPLconst} для нормы точного градиента. Поэтому рассматривается альтернативный вариант выбора шага для метода \eqref{eq:PL:method} с относительной погрешностью задания градиента, который позволит получить приемлемый аналог оценки  \eqref{eq:PL:funcPLconst} для квадрата нормы неточного градиента. Аналогично ранее рассмотренным случаям, пользуясь $L$-липшицевостью градиента и условием \eqref{eq:PL:rel_err}, можно вывести следующее неравенство \cite{puchinin2023gradient}
\begin{multline}\label{eq:PL:cond}
    f(x^{k + 1}) \leq f(x^k) + \left\langle \Tilde{\nabla} f(x^k), x^{k + 1} - x^k \right\rangle + \frac{L}{2} \left\lVert x^{k + 1} - x^k \right\rVert_2^2 +\\+ \frac{\alpha}{1 - \alpha} \left\lVert \Tilde{\nabla} f(x^k) \right\rVert_2 \left\lVert x^{k + 1} - x^k \right\rVert_2.
\end{multline}
Таким образом, при $\alpha \in [0; 0.5)$ для адаптивного варианта градиентного метода \eqref{eq:PL:method} c 
\begin{equation}\label{adaptiv_stp_rl}
h_k = \frac{1}{L_{k+1}}\frac{1-2\alpha}{1-\alpha}
\end{equation}
и критерием выхода из итерации (до его выполнения $L_{k+1}$ увеличивается в 2 раза)
\begin{multline}
    f(x^{k + 1}) \leq f(x^k) + \left\langle \Tilde{\nabla} f(x^k), x^{k + 1} - x^k \right\rangle + \frac{L_{k+1}}{2} \left\lVert x^{k + 1} - x^k \right\rVert_2^2 +\\+ \frac{\alpha}{1 - \alpha} \left\lVert \Tilde{\nabla} f(x^k) \right\rVert_2 \left\lVert x^{k + 1} - x^k \right\rVert_2
\end{multline}
неравенство \eqref{eq:PL:cond} гарантирует выход из итерации при $L_{k+1} \geq L$. Тогда оценка скорости сходимости метода \eqref{eq:PL:method} c шагом \eqref{adaptiv_stp_rl} имеет вид
\begin{equation}\label{eq:PL:cond:adaptiv:summ}
    \begin{aligned}
        f(x^N) - f(x_*) &\leq \prod\limits_{k=0}^{N-1} \left(1 - \frac{\mu}{L_{k+1}} (1 - 2\alpha)^2 \right)( f(x^0) - f(x_*)).
    \end{aligned}
\end{equation}
Если $L_0 \leq 2L$, то последнее неравенство означает сходимость рассматриваемого метода со скоростью геометрической прогрессии
$$
f(x^N) - f(x_*) \leq\left(1 - \frac{\mu}{2L} (1 - 2\alpha)^2 \right)^N\left( f(x^0) - f(x_*) \right).
$$
Отметим, что оценка \eqref{eq:PL:cond:adaptiv:summ} может быть применима в случае неизвестного значения $L$. Кроме того, даже в случае известной оценки на $L$ потенциально возможно улучшение качества точки выхода метода при условии $L_{k+1} < L$.

\subsection{Метод условного градиента (Франк--Вульфа--Левитина--Поляка)}

В этом подпункте мы вспомним о направлении научных исследований Б.Т. Поляка, связанном с методами условного градиента. К примеру, таким методам посвящена достаточно известная работа \cite{левитин1966методы}.  
Несмотря на то, что методы условного градиента по числу итераций проигрывают оптимальным на классе выпуклых гладких задач ускоренным методам (см. раздел \ref{sec:accelerated}) для решения задач со структурой (например, на единичном симплексе или прямом произведении таких симплексов), по времени работы эти методы могут быть существенно эффективнее ускоренных из-за возможной дешевизны итерации \cite{bubeck2015convex,cox2017decomposition,гасников2020модели,anikin2022efficient}. На данный момент методы условного градиента достаточно хорошо разработаны, см., например, обзорную статью \cite{bomze2021frank} и монографию \cite{braun2022conditional}. В частности, среди последних достижений можно отметить безградиентные варианты метода условного градиента для задач стохастической оптимизации \cite{lobanov2023zero}, а также вариант метода условного градиента для децентрализованных оптимизационных задач на меняющихся графах \cite{vedernikov2023decentralized}.

Рассмотрим задачу выпуклой оптимизации ($f(x)$ -- выпуклая функция, $Q$ -- ограниченное выпуклое множество):
\begin{equation}\label{eq:f_Q}
    \min_{x\in Q} f(x).
\end{equation}
Обозначим решение задачи \eqref{eq:f_Q}  через $x_*$ (если решение не единственно, то $x_*$ -- какое-то из решений).

\floatname{algorithm}{Алгоритм}
	\begin{algorithm}
		\caption{Метод условного градиента \cite{левитин1966методы,braun2022conditional}}\label{ch2_alg_cond_grad}
		\begin{algorithmic}[1]
			\REQUIRE $f$~--- выпуклая, непрерывно дифференцируемая функция; $Q$~--- допустимое множество, выпуклое и компактное; $x^{1}$~--- начальная точка; $N$~--- количество итераций.
			\ENSURE точка $x^N$
%\STATE 			

\FOR{$k=1,\, \dots, \, N-1$}
\STATE $\gamma_k = \frac{2}{k+1}$, $0 \leq \gamma_k \leq 1$

\STATE $y^{k} = \argmin_{y \, \in \, Q} \langle \nabla f(x^k), \, y \rangle$ 

\STATE $x^{k + 1} = (1 - \gamma_k) x^k + \gamma_k y^k$
\ENDFOR

\RETURN $x^N$
		\end{algorithmic}

\end{algorithm}

\begin{theorem}[см. \cite{левитин1966методы,braun2022conditional}]\label{th:FW}
Пусть $\nabla f(x)$ удовлетворяет на $Q$ условию Липшица с константой~$L$ по отношению к некоторой норме $\| \cdot \|:$ для всех $x,y \in Q$ выполняется
$$\|\nabla f(y) - \nabla f(x)\|_* \le L\|y-x\|,$$
$R = \sup_{x,\, y \, \in \, Q} \| x - y \|$, $\gamma_k = \frac{2}{k+1}$ для $k \geq 1$.
Тогда для любого $N \geq 2$ выполняется
\begin{equation}
\label{eq:ch2_eq_rate_FW}
f(x^N) - f(x_*) \leq \frac{2 L R^2}{N + 2}.
\end{equation}
\end{theorem}

\begin{proof} Из условия Липшица на градиент $f(x)$ и выпуклости $f(x)$ имеем оценку
\[ 0 \leq f(x) - f(y) - \langle \nabla f(y),x-y \rangle \leq \frac{L}{2}\|x - y\|^2
\]
для всех $x,y \in Q$. Получаем цепочку неравенств:
\begin{align*}
\lefteqn{f(x^{k+1}) - f(x^k) \leq \langle \nabla f(x^k),x^{k+1}-x^k \rangle + \frac{L}{2}\|x^{k+1} - x^k\|^2} \phantom{space} \\ 
=& \,\gamma_k \langle \nabla f(x^k),y^k-x^k \rangle + \frac{L\gamma_k^2}{2}\|y^k - x^k\|^2\\ \leq& \, \gamma_k \langle \nabla f(x^k),x_*-x^k \rangle + \frac{L\gamma_k^2}{2}R^2  \leq \,\gamma_k(f(x_*) - f(x^k)) + \gamma_k^2\frac{LR^2}{2}.
\end{align*}
Введём обозначение $\delta_k = \frac{f(x^k) - f(x_*)}{LR^2}$. 
Тогда неравенство можно переписать в виде
$$ \delta_{k+1} \leq (1-\gamma_k)\delta_k + \frac{\gamma_k^2}{2} = \left( 1 - \frac{2}{k+1} \right)\delta_k + \frac{2}{(k+1)^2}.
$$
Начиная с неравенства 
$$\delta_{2}  \leq \left( 1 - \frac{2}{1+1} \right)\delta_1 + \frac{2}{(1+1)^2} = \frac{1}{2},$$
применением индукции по $k$ получаем желаемый результат. \end{proof}

Как видим из предыдущего результата, при получении оценки скорости сходимости метода условного градиента не важен выбор конкретной нормы. Важно, чтобы параметры $L$ и $R$ оценивались относительно согласованных норм $\|\cdot\|$ и $\|\cdot\|_*$.

Следуя \cite{поляк1983введение}, покажем, что оценка \eqref{eq:ch2_eq_rate_FW} не может быть улучшена (с точностью до числового множителя). 
Для этого выберем $f(x_1,x_2) = x_1^2 + (1+x_2)^2$, $Q = \left\{x:~|x_1|\le1, 0\le x_2\le 1\right\}$. Тогда $x_* = (0,0)^T$. 
При этом $y^k = (1,0)^T$ при $y_1^k <0$ и $y^k = (-1,0)^T$ при $y_1^k > 0$. 
Несложно показать, что для этого примера $f(x^N) - f(x_*) \simeq \frac{4}{N}$. 
Причем эта ситуация типична для случая, когда $Q$ --- многогранник, а минимум достигается в одной из его вершин, поскольку в качестве $y^k$ могут выступать лишь вершины $Q$. 

Отметим также, что в рассматриваемой работе Левитина--Поляка \cite{левитин1966методы} впервые для метода условного градиента появилась идея использовать параметр шага, зависящий от константы Липшица градиента целевой функции $L$. В современных работах \cite{bomze2021frank, aivazian2023adaptive, braun2022conditional} такой шаг обычно применяют в виде
\begin{equation}\label{eq:FW:eq8}
    \gamma_k = \gamma_k (L): = \min \left \{-\frac{\nabla f (x^k)^T (y^k - x^k)}{L \| y^k - x^k \|^2}, 1 \right \},
\end{equation}
где $L$~--- константа Липшица $\nabla f$. Размер соответствующего шага может рассматриваться как результат минимизации квадратичной модели $ m_k (\cdot\,; L)$, переоценивающей $f$ вдоль прямой $x^k + \gamma_k (y^k - x^k) $,
\begin{equation}\label{eqFWLem2}
\begin{aligned}
    m_k (\gamma_k; L) & = f (x^k) + \gamma_k \nabla f (x^k)^T (y^k - x^k)  + \frac{L \gamma_k^2}{2} \|y^k - x^k \|^2 \geq f (x^k + \gamma_k(y^k - x^k)),
\end{aligned}
\end{equation}
где последнее неравенство следует из стандартной леммы о спуске. Интерес такого выбора шага для метода Франк--Вульфа заключается, в частности, в возможности предложить достаточные условия убывания на итерации невязки по функции не менее чем в 2 раза. Точнее говоря, согласно лемме 2 из \cite{bomze2021frank} для всякой выпуклой функции $f$ в случае, если на $k$-й итерации в \eqref{eqFWLem2} верно $\gamma_k = 1$, то верно неравенство
\begin{equation}\label{eqFWLm2}
        f(x^{k+1}) - f^* \leq \frac{1}{2} \min \left(L\|y^k - x^k\|^2, f(x^k) - f^*\right).
    \end{equation}

При введении дополнительного условия на функцию, например градиентного доминирования, для шага вида \eqref{eq:FW:eq8} можно получить и результаты о сходимости метода типа условного градиента со скоростью геометрической прогрессии. Например, известен такой результат для выпуклых функций с условием Поляка--Лоясиевича \cite{bomze2021frank}. При этом такие функции могут не быть сильно выпуклыми (например, целевая функция из задачи логистической регрессии не сильно выпукла, но удовлетворяет условию градиентного доминирования на всяком компакте \cite{karimi2016linear}).
\begin{theorem}\label{thmlingradominFW}
    Пусть $Q$~--- компактное выпуклое множество диаметра $R$, а $f$~--- выпуклая функция, удовлетворяющая условию градиентного доминирования с константой $\mu >0$. Пусть также существует точка минимума $f$ и она лежит во внутренности множества $Q$ ($x_* \in Int(Q)$), т.~е. существует $r > 0$, такой, что $B(x_*, r) \subset Q$. Тогда при любых $k \geq 1$ верно
    \begin{equation*}
        f(x^k) - f^* \leq  (f(x^0)-f^*) \prod\limits_{i=1}^k \varphi_i,
    \end{equation*}
    где
    \begin{equation*}
        \varphi_i =
        \begin{cases}
           \frac{1}{2}, & \text{ если } \alpha_i = 1,
           \\
           1 - \frac{r^2}{L c^2 R^2}, & \text{ если } \alpha_i < 1,
        \end{cases}
    \end{equation*}
и $c^2 = \frac{1}{2 \mu}$. Таким образом, имеет место сходимость по функции со скоростью геометрической прогрессии, знаменатель которой в худшем случае равен $\max \left\{\frac{1}{2}, 1 - \frac{r^2}{L c^2 R^2}\right\}$.
\end{theorem}
Отметим, что недавно в работе \cite{aivazian2023adaptive} доказан аналог теоремы \ref{thmlingradominFW} для варианта метода условного градиента с адаптивным подбором шага.

Что касается метода условного градиента, исследованного в самой работе Левитина--Поляка \cite{левитин1966методы}, то обратим внимание на две отмеченные там его особенности. Первая из них касается нижней оценки скорости сходимости такого метода. Точнее говоря, в замечании 1 на стр. 805 \cite{левитин1966методы} отмечено, что оценка $O\left(\frac{1}{k}\right)$ невязки по функции неулучшаема с точностью до константы на классе выпуклых гладких функций $f$. Для того времени рассуждения на тему нижних оценок были редки, поскольку направление исследований численных методов оптимизации, основанное на нижних оценках их эффективности для классов задач, зародилось несколько позднее в монографии \cite{немировский1979сложность}. Вторая особенность метода условного градиента, отмеченная впервые в \cite{левитин1966методы}, касается повышения скоростных гарантий не путём накладывания дополнительных условий на целевую функцию, а посредством сужения класса допустимых множеств задачи $Q$. Например, при добавлении требования, что множество $Q$ является сильно выпуклым. Множество $Q$ называется сильно выпуклым, если существует функция $\delta(\tau) = \gamma \tau^2$ ($\gamma > 0$) такая, что $(x + y) / 2 + z \in Q$ для любых $x,y \in Q$ и любого $z$: $\lVert z \rVert \leq \delta\left(\lVert x - y \rVert\right)$. Существуют эквивалентные определения сильной выпуклости множества. Например, как было показано в \cite{vial1982strong} (теорема~1 на стр.~190), множество является сильно выпуклым тогда и только тогда, когда оно представимо в виде пересечения шаров одного радиуса. Более подробно см. на стр.~789~\cite{левитин1966методы}, а также в \cite{vial1982strong, Vial, половинкин1996сильно}, метод сходится по аргументу со скоростью геометрической прогрессии при условии, что $\lVert \nabla f(x) \rVert \geq \varepsilon > 0$ на $Q$.

Исследование разных вариантов метода условного градиента продолжается и в настоящее время \cite{bomze2021frank, aivazian2023adaptive, braun2022conditional, ito2023parameter}. В работе \cite{ito2023parameter} имеется обзор результатов последних лет по оценкам сложности для метода Франк--Вульфа, предложен универсальный (адаптивный по параметру гладкости задачи) вариант метода для задач с равномерно выпуклой структурой. В \cite{aivazian2023adaptive} адаптивный вариант метода Франк--Вульфа с шагом \eqref{eq:FW:eq8} исследуется для выпуклых задач (вообще говоря, без равномерной структуры), причём детально проработан соответствующий аналог оценки \eqref{eqFWLm2}.

\subsection{Ускоренные методы выпуклой оптимизации}\label{sec:accelerated}
Для наглядности изложения рассмотрим задачу безусловной выпуклой оптимизации 
\begin{equation}\label{eq:f}
    \min_x f(x),
 \end{equation}
 в которой выполняется условие $L$-Липшицевости градиента в $2$-норме, 
  см. \eqref{eq:PL:Lcond}, 
 а $x$ -- принадлежит гильбертову пространству. Также для наглядности будем опускать числовые множители в оценках скорости сходимости. 
 
 В кандидатской диссертации Б.Т. Поляка в 1963 году (см. также \cite{поляк1964градиентные})  было показано, что градиентный спуск на такой задаче
 $$x^{k+1} = x^k - \frac{1}{L}\nabla f(x^k)$$
 будет сходиться следующим образом 
 \begin{equation*}
 f(x^N) - f(x_*) \lesssim \frac{LR^2}{N},
  \end{equation*}
где $R = \|x^0 - x_*\|_2$ -- расстояние от точки старта до ближайшего (в $2$-норме) к точке старта решения $x_*$ задачи \eqref{eq:f} (не ограничивая общности, везде в разделе~\ref{sec:accelerated}  можно считать, что все константы, типа $L$ и $\mu$, определены на шаре с центром в точке $x^0$ и радиусом $2R$ \cite{nesterov2018lectures,гасников2018современные}). Оптимальная ли эта оценка? Оказывается, что даже если выбирать шаг метода $h = 1/L$ как-то по-другому, улучшить такую оценку на рассматриваемом классе задач можно только на числовой множитель порядка $4$ \cite{taylor2017smooth}. Однако это совсем не означает, что если рассмотреть какой-то другой метод градиентного типа, то на нем также нельзя будет получить существенно лучшую оценку скорости сходимости. Действительно, как минимум на квадратичных задачах выпуклой оптимизации метод сопряженных градиентов (первая итерация совпадает с итерацией типа градиентного спуска, аналогично и для приводимых далее моментных методов):
\begin{equation}\label{eq:CG}
x^{k+1} = x^k - \alpha_k\nabla f(x^k) + \beta_k (x^k - x^{k-1}),
\end{equation}
где 
$$\left(\alpha_k,\beta_k\right) \in \text{Arg}\min_{\alpha,\beta} f\left(x^k - \alpha\nabla f(x^k) + \beta (x^k - x^{k-1})\right)$$ 
дает существенно лучшую оценку скорости сходимости (при дополнительных предположениях о распределении спектра матрицы квадратичной формы оценка может быть дополнительно улучшена \cite{goujaud2022super}):
\begin{equation}\label{eq:accel}
     f(x^N) - f(x_*) \lesssim \frac{LR^2}{N^2}.
\end{equation}
Наиболее тонкое исследование скорости сходимости метода сопряженных градиентов \eqref{eq:CG} (в том числе в условиях неточностей) имеется в работе \cite{немировский1986регуляризующих}. 
 Стоимость итерации такого метода (в виду возможности решить задачу поиска $\alpha_k$ и $\beta_k$ аналитически для квадратичных задач) будет по порядку такой же, как  стоимость итерации градиентного метода. Аналогичные оценки можно написать и для $\mu$-сильно выпуклых задач:
\begin{equation}\label{eq:non_accel_mu}
    f(x^N) - f(x_*) \lesssim LR^2\exp\left(-\frac{\mu}{L}N\right)
\end{equation}
  для градиентного спуска и:
    \begin{equation}\label{eq:accel_mu}
    f(x^N) - f(x_*) \lesssim LR^2\exp\left(-2\sqrt{\frac{\mu}{L}}N\right)
    \end{equation}
 для метода сопряженных градиентов \eqref{eq:CG}  на квадратичной задаче.

Заметим, что одна из форм записи метода сопряженных градиентов \eqref{eq:CG} в случае $\mu$-сильно выпуклой целевой функции приводит к методу Чебышёва \cite{d2021acceleration}:
\begin{equation}\label{eq:Chebyshev}
x^{k+1} = x^k - \frac{4\delta_k}{L-\mu}\nabla f(x^k) + \left( \frac{2\delta_k\left(L+\mu\right)}{L-\mu} - 1\right)\left(x^{k} - x^{k-1}\right),
\end{equation}
$$x^1 = x^0 - \frac{2}{L+\mu}\nabla f(x^0),$$
$$\delta_{k+1} = \frac{1}{2\frac{L+\mu}{L-\mu} - \delta_k}, \quad \delta_1 = \frac{1}{2\frac{L+\mu}{L-\mu} + 1}.$$
Метод тяжёлого шарика Поляка (или импульсный / моментный метод Поляка)  \cite{поляк1964некоторых} при этом выглядит так:
\begin{equation}\label{eq:HB}
x^{k+1} = x^{k} - \frac{4}{\left(\sqrt{L}+\sqrt{\mu}\right)^2}\nabla f(x^k) + \frac{\left(\sqrt{L}-\sqrt{\mu}\right)^2}{\left(\sqrt{L}+\sqrt{\mu}\right)^2}\left(x^k - x^{k-1}\right).
\end{equation}
Этот метод получается в асимптотике $k\to\infty$ из метода Чебышёва \eqref{eq:Chebyshev}, поскольку 
$$\delta_{\infty} = \frac{1}{2\frac{L-\mu}{L+\mu}-\delta_{\infty}},$$
следовательно,
$$\delta_{\infty} = \frac{\sqrt{L} - \sqrt{\mu}}{\sqrt{L} + \sqrt{\mu}}.$$
Приведенный вывод (с учетом оптимальности метода Чебышёва на квадратичных задачах) дает надежду, что метод тяжелого шарика \eqref{eq:HB} в некотором смысле асимптотически оптимальный. Оказывается, что <<в среднем>> так оно и есть \cite{scieur2020universal}.

Исходно метод тяжелого шарика \eqref{eq:HB}  был получен с помощью первого метода Ляпунова \cite{поляк1983введение}, 
который сводит анализ скорости локальной сходимости метода  на классе выпуклых задач к анализу скорости глобальной сходимости для квадратичных выпуклых задач. Б.Т. Поляк подбирал коэффициенты $\alpha$ и $\beta$ постоянными. 
Подобранные коэффициенты гарантировали локальную скорость сходимости метода \eqref{eq:HB}  аналогичную скорости сходимости метода сопряженных градиентов \eqref{eq:accel_mu}. Отметим, что при этом имеется некоторая физическая аналогия в полученном таким образом методе \eqref{eq:HB} 
 с овражным методом Гельфанда--Цетлина
\cite{гельфанд1961принцип}. Однако как показал последующий анализ метода тяжелого шарика \eqref{eq:HB}, в общем случае нельзя гарантировать его глобальную сходимости \cite{lessard2016analysis}. Ее можно добиться за счет некоторых дополнительных предположений (о гладкости), но скорость глобальной сходимости получалась уже не лучше, чем у обычного градиентного спуска \cite{ghadimi2015global}, причем то что лучше и не получится удалось доказать \cite{goujaud2023provable}. Также на примере тяжелого шарика \eqref{eq:HB}  хорошо демонстрируется общая особенность ускоренных методов -- немонотонное убывание целевой функции по ходу итерационного процесса и наличие циклов длин $\sim\sqrt{L/\mu}$ (после каждого цикла невязка по функции убывает в $\sim 2$ раза). Изящное объяснение этому явлению можно найти в работах \cite{o2015adaptive,danilova2020non}.

 Метод тяжелого шарика \eqref{eq:HB} сыграл очень важную роль в развитии ускоренных методов выпуклой оптимизации. Общее направление исследований тут можно описать как попытку сделать такую версию метода сопряженных градиентов, для которой бы удалось доказать глобальную оценку скорости сходимости, аналогичную приведенным выше, для обычного метода сопряженных градиентов на квадратичных задачах. Так родились, например, методы Флетчера--Ривса, Полака--Рибьера--Поляка \cite{поляк1969метод}. Точку здесь удалось поставить А.С. Немировскому в конце 70-х годов прошлого века \cite{немировский1979сложность,немировский1982орт}, попутно показав, что полученные оценки скорости сходимости \eqref{eq:accel} и \eqref{eq:accel_mu} уже не могут быть дальше улучшены никакими другими методами, использующими градиент целевой функции (улучшить немного можно лишь числовой множитель). То есть получилось, что сложность класса гладких задач (сильно) выпуклой оптимизации с точностью до числовых множителей совпадает со сложностью аналогичных классов задач (сильно) выпуклой квадратичной оптимизации. Отметим, что для распределенной оптимизации (а точнее, федеративного обучения) недавно было обнаружено, что такая аналогия уже не имеет места \cite{woodworth2020local,woodworth2021min}. 
Отметим также, что ускоренные методы А.С. Немировского (известно как минимум три таких метода) требовали на каждой итерации решение вспомогательной маломерной задачи оптимизации. Избавиться от этого удалось в кандидатской диссертации Ю.Е. Нестерова (научным руководителем был Б.Т. Поляк) в 1983 году \cite{нестеров1983метод}:
\begin{equation}\label{eq:Nesterov}
x^{k+1} = x^{k} - \frac{1}{L}\nabla f\left(x^k + \frac{\sqrt{L}-\sqrt{\mu}}{\sqrt{L}+\sqrt{\mu}}\left(x^k - x^{k-1}\right)\right) + \frac{\sqrt{L}-\sqrt{\mu}}{\sqrt{L}+\sqrt{\mu}}\left(x^k - x^{k-1}\right).
\end{equation}
Такой ускоренный (моментный) метод Нестерова будет сходиться аналогично \eqref{eq:accel_mu} уже глобально на классе гладких $\mu$-сильно выпуклых задач (аналогично можно предложить вариант метода и для выпуклых задач $\frac{\sqrt{L}-\sqrt{\mu}}{\sqrt{L}+\sqrt{\mu}}\to\frac{k-1}{k+2}$ в \eqref{eq:Nesterov}). Отметим, что недавно были предложены (см. обзор \cite{d2021acceleration}) такие варианты ускоренного метода Нестерова (с аналогичными по порядку трудозатратами на каждой итерации), для которых удалось доказать, что их скорость сходимости в выпуклом и сильно выпуклых случаях оптимальны (без оговорки, <<с точностью до числового множителя>>) на классе градиентных методов выпуклой оптимизации, использующих всевозможные линейные комбинации полученных на предыдущих итерациях градиентов (в том числе, допускается вспомогательная минимизация на подпространствах натянутых на полученные градиенты). Например, в $\mu$-сильно выпуклом случае оптимальный алгоритм выглядит таким образом.
\floatname{algorithm}{Алгоритм}
	\begin{algorithm}
		\caption{Ускоренный метод Тейлора--Дрори \cite{d2021acceleration}}\label{Taylor}
		\begin{algorithmic}[1]
			\REQUIRE $f$~--- $\mu$-сильно выпуклая функция с $L$-липшицевым градиентом, точка старта $x^0$
% 			\ENSURE точка $z^N$
\STATE \textbf{SET:} $z^0=x^0$, $A_{0} = 0$, $q=\mu/L$
\FOR{$k=0, \, \dots, \, N-1$}
\STATE $A_{k+1}=\frac{(1+q)A_k+2\left(1 + \sqrt{(1+A_k)(1+qA_k)}\right)}{(1-q)^2}$
\STATE $\tau_k = 1 - \frac{A_k}{(1-q)A_{k+1}}$ и $\delta_k = \frac{1}{2}\frac{(1-q)^2A_{k+1} - (1+q)A_k}{1+q+qA_k}$
\STATE $y^k = x^k + \tau_k\left(z^k - x^k\right)$
\STATE $x^{k+1} = y^{k} - \frac{1}{L}\nabla f(y^k)$
\STATE $z^{k+1} = (1 - q\delta_k)z^k + q\delta_k y^k - \frac{\delta_k}{L}\nabla f(y^k)$
\ENDFOR
\RETURN $z^N$
\end{algorithmic}
\end{algorithm}

Как ни странно, в 1983 году статья Ю.Е. Нестерова не вызвала какого-то большого ажиотажа. О ней вспомнили лишь спустя 20 лет уже в новом столетии, когда размеры новых задач, которые требовалось решать в анализе данных, позволяли использовать только градиентные методы (а не, скажем, методы внутренней точки) или их стохастические варианты. Во многом этому способствовало появление в 2004 году первого издания лекций Ю.Е. Нестерова по выпуклой оптимизации \cite{nesterov2018lectures}, в которых центральную роль как раз и заняло современное (на тот момент) изложение ускоренных методов. Собственно, вот уже 20 лет идет настоящий бум ускоренных методов, см., например, \cite{nesterov2018lectures,lan2020first,lin2020accelerated}. Новые тонкие результаты о сходимости таких методов появляются и по сей день, см., например, \cite{peng2023nesterov}. Отметим, что ускоренные методы предложены для задач с ограничениями, с более общим понятием проектирования (неевклидова), для задач со структурой (в условиях модельной общности), в условиях неточности градиента и неточности проектирования, см.
\cite{devolder2013exactness,stonyakin2021inexact,zhang2022solving,vasin2023accelerated} и цитированную тут литературу. Например, в малоизвестной работе Б.Т. Поляка \cite{poljak1981iterative} было показано, что при наличии малого (сколь угодно малого, но фиксированного по масштабу) аддитивного враждебного шума в градиенте \cite{поляк1983введение}
$$\| \tilde \nabla f(x) - \nabla f(x)\|_2 \le \delta$$ 
нельзя гарантировать сходимость градиентного спуска. Более того, можно привести пример, когда он будет расходиться. Однако за счет ранней остановки можно решить данную проблему (см. также работу \cite{немировский1986регуляризующих}, в которой обсуждается ранняя остановка для метода сопряженных градиентов в условиях неточной информации). В частности, аналогичный результат имеет место и для ускоренных методов. Грубо говоря, можно показать, что по аналогии с \eqref{eq:accel} имеет место:
$$f(x^N) - f(x_*) \lesssim \frac{LR^2}{N} + R\delta$$
до тех пор (для таких $N$) пока первое слагаемое не станет меньше второго. В момент, когда это произойдет нужно останавливать метод, иначе он может уже начать расходиться. При этом если концепция шума относительная \cite{поляк1983введение}
$$\|\tilde \nabla f(x) - \nabla f(x)\|_2 \le \alpha\|\nabla f(x)\|_2,$$ 
где $\alpha \in [0,1)$, то в отличие от неускоренных методов (см. раздел \ref{sec:PL}) для ускоренных возможность сохранения порядка скорости сходимости остается только при достаточно малых $\alpha$ для сильно выпуклых задач и должным образом убывающих на итерациях $\alpha_k$ для выпуклых задач \cite{vasin2023accelerated,kornilov2023intermediate}.

Сложно переоценить роль ускорения в современной оптимизации. Практически все современные методы решения задач больших размерностей (в том числе для многих невыпуклых задач) используют ускорение. Ускорение можно дополнительно структурировать. Например, задачу вида  
$$\min_x~f(x) + g(x),$$
в случае когда $f(x)$ -- $\mu_f$-сильно выпуклая и $L_f$-гладкая, а $g(x)$ -- $\mu_g$-сильно выпуклая и $L_g$-гладкая, можно решить с относительной точностью $\varepsilon$ за $\mathcal{O}\left(\sqrt{L_f/(\mu_f + \mu_g)}\ln(1/\varepsilon)\right)$ вычислений $\nabla f$ и $\mathcal{O}\left(\sqrt{L_g/(\mu_f + \mu_g)}\ln(1/\varepsilon)\right)$ вычислений $\nabla g$ \cite{lan2020first,kovalev2022optimal}. Если использовать стандартное ускорение (не учитывая структуру задачи), то согласно \eqref{eq:accel_mu} удалось бы получить только такой результат: $\mathcal{O}\left(\sqrt{(L_f + L_g)/(\mu_f + \mu_g)}\ln(1/\varepsilon)\right)$ вычислений $\nabla f$ и $\nabla g$, что, очевидно, может быть хуже.

Задачу вида  
 $$\min_{x,y} f(x,y),$$
в случае когда $f(x,y)$ -- выпуклая функция по совокупности аргументов, $\mu_x$-сильно выпуклая и $L_x$-гладкая по $x$, а также $\mu_y$-сильно выпуклая и $L_y$-гладкая по $y$, можно решить с относительной точностью $\varepsilon$ за $\mathcal{O}\left(\sqrt{L_x/\mu_x}\ln(1/\varepsilon)\right)$ вычислений $\nabla_x f$ и $\mathcal{O}\left(\sqrt{L_y/\mu_y }\ln(1/\varepsilon)\right)$ вычислений $\nabla_y f$ \cite{kovalev2022optimal_minmin}.  
Если использовать стандартное ускорение (не учитывая структуру задачи), то (при некоторых дополнительных оговорках) согласно \eqref{eq:accel_mu} удалось бы получить только такой результат: $\mathcal{O}\left(\sqrt{\max\left\{L_x,L_y\right\}/\min\left\{\mu_x,\mu_y\right\}}\ln(1/\varepsilon)\right)$ вычислений $\nabla_x f$ и $\nabla_y f$, что также может быть хуже. 

 Более того, недавно были получены нетривиальные варианты ускоренных методов для седловых задач со структурой \cite{borodich2023optimal}. Например, если рассматривается такая задача 
 $$\min_{x}\max_{y}~p(x) + F(x,y) - q(y),$$
 где $p(x),~q(y)$ -- выпуклые и $L_p,~L_q$-гладкие функции, а $F(x, y)$ -- $L_F$-гладкая, $\mu_x$-сильно выпуклая и $\mu_y$-сильно вогнутая, то существует такой ускоренный метод, который решит задачу с относительной точностью $\varepsilon$ по зазору двойственности \cite{beznosikov2023smooth} за $$\mathcal{O}\left(\left(\sqrt{\frac{L_p}{\mu_x}} + \sqrt{\frac{L_q}{\mu_y}}\right) \log \frac{1}{\varepsilon}\right)$$ вычислений $\nabla p(x)$, $\nabla q(y)$ и $$\mathcal{O} \left( \max\left\{\sqrt{\frac{L_p}{\mu_x}}, \sqrt{\frac{L_q}{\mu_y}}, \frac{L_F}{\sqrt{\mu_x \mu_y}} \right\}\log \frac{1}{\varepsilon}\right)$$ вычислений $\nabla F(x,y)$.

 Начиная с пионерской работы Ю.Е. Нестерова и Б.Т. Поляка 2006 года \cite{nesterov2006cubic} в нескольких ведущих мировых центрах по оптимизации стали активно  разрабатываться тензорные методы (методы, использующие старшие производные). В частности, в работе \cite{nesterov2021implementable} было показано, что при естественных условиях тензорные методы второго и третьего порядка (использующие производные второго и третьего порядка) могут быть реализованы практически с той же стоимостью итерации, как у метода Ньютона. Оптимальные (с точностью до числовых множителей) ускоренные  тензорные методы были предложены в работах \cite{monteiro2013accelerated,nesterov2018lectures,gasnikov2019near,kovalev2022first,NEURIPS2022_7ff97417,kamzolov2022exploiting}. 
 Например, оптимальный тензорный метод, использующий производные порядка $r\ge 2$, при условии, что целевая функция $\mu$-сильно выпуклая и тензор $r$-х производных $M_r$-Липшицев, для достижения относительной точности $\varepsilon$ требует:
 $$\mathcal{O} \left(\left(\frac{M_rR^{r-1}}{\mu}\right)^{\frac{2}{3r+1}}+ \ln\ln\left(\frac{1}{\varepsilon}\right)\right)$$
 итераций (вычислений старших производных). При этом при $r = 2,~3$ трудозатратность каждой итерации оптимальных методов практически такая же как и у метода Ньютона. Отметим, что итерационная (оракульная) сложность (необходимое число итераций) во втором слагаемом отвечает локальной итерационной сложности метода Ньютона. Проблема с методом Ньютона только в том, что нужно оказаться в окрестности квадратичной скорости сходимости. Первое слагаемое в приведенной оценке как раз отвечает за время попадания в эту окрестность. Обратим внимание, что даже если доступны старшие производные высокого порядка ($r$ -- большое) и целевая функция достаточно гладкая, то второе слагаемое остается по порядку неизменным. Это означает, что локальная скорость сходимости метода Ньютона по порядку оптимальна в классе всевозможных численных методов оптимизации \cite{немировский1979сложность}.  

 Из написанного выше может создаться впечатление, что ускорение возможно заточить под структуру любой гладкой задачи. В разделе \ref{sec:PL} отмечалось, что если вместо (сильной) выпуклости имеет место условие Поляка--Лоясиевича, то в общем случае ускорение невозможно. Кажется, что если ограничиться только (сильно) выпуклыми постановками, то ускорение всегда можно сделать. Как правило, так оно и есть. Но имеются такие постановки выпуклых задач, в которых доказано, что ускорение в общем случае также невозможно. К таким примерам относится задача получения для гладких выпуклых задач оптимальных (по числу коммуникаций и оракульных вызовов --- вычислений градиентов) децентрализованных ускоренных методов на меняющихся графах. Децентрализованная оптимизация \cite{bertsekas2015parallel,gorbunov2022recent} является частью распределенной оптимизации. Отметим, что одна из первых работ по распределенной оптимизации была  у аспиранта Б.Т. Поляка в конце 70-х годов прошлого века \cite{кибардин1979декомпозиция}. Бурно развивающийся подраздел распределенной оптимизации -- децентрализованная оптимизация на меняющихся со временем графах \cite{rogozin2022decentralized}. Оказывается (см. \cite{metelev2023decentralized}), что в оценке числа коммуникаций невозможно в общем случае ускорение потому как входит худшее (на итерациях) число обусловленности коммуникационного графа (если бы граф не менялся ускорение было бы возможно за счет ускоренного консенсуса -- Чебышёвское ускорение). Несмотря на приведенный пример, все же во всех известных современных постановках выпуклых гладких задач (с различной структурой), как правило, удается добиться ускорения, учитывающего данную структуру. Причем, основные продвижения здесь были получены буквально в последние десять--пятнадцать лет.

 Также может показаться, что уже не осталось открытых вопросов, и на все вопросы для задач гладкой выпуклой оптимизации получены ответы. На самом деле, это не так. Например, в упомянутой ранее седловой задаче со структурой открытым остается вопрос о возможности дополнительного разделения оракульных сложностей по числу вызовов $\nabla p(x)$ и $\nabla q(x)$. Но можно не опускаться в такие частности и заметить, что даже для самой обычной задачи гладкой сильно выпуклой оптимизации до сих пор не предложен ускоренный метод, адаптивный по всем параметрам $\mu,~L$. При этом в классе неускоренных методов наискорейший спуск решает данную проблему, см. также \cite{bao2023global}. К сожалению, адаптивный метод сопряженных градиентов в форме \eqref{eq:CG}, как уже отмечалось, не обязан сходиться с желаемой скоростью на классе (сильно) выпуклых гладких задач. И хотя сейчас есть  кандидаты (адаптивные ускоренные методы), которые, возможно, обладают желаемыми свойствами \cite{guminov2023accelerated}, однако пока это не удалось строго доказать.

\section{Методы стохастической оптимизации}\label{sec:stoch}
Задачей стохастической оптимизации называется задача вида
\begin{equation}\label{SO}
  \min\limits_{x\in Q} f(x):=\EE_{\xi} f(x,\xi),  
\end{equation}
в которой можно в любой точке $x$ вызвать оракул (подпрограмму), выдающий $\nabla f(x,\xi)$ с новой независимой реализацией $\xi$ (допускается вызов оракула в одной точке $x$ многократно -- см. далее батчирование). При этом $\EE_{\xi} \nabla f(x,\xi) \equiv \nabla f(x)$. Целью является найти $\varepsilon$-приближенное решение задачи \eqref{SO} (по функции $f(x)$) за наименьшее число вызовов оракула. Без преувеличения можно сказать, что современный анализ данных в алгоритмической своей части -- это решение соответствующих задач стохастической оптимизации \cite{beznosikov2024}.

\subsection{Стохастический градиентный спуск}\label{sec:SGD}
По аналогии с градиентным спуском для решения \eqref{SO}, естественно, рассматривать, так называемые, стохастические градиентные спуски (SGD) \cite{robbins1951stochastic,ермольев1976методы}:
\begin{equation}\label{eq:projSGD}
x^{k+1}= \pi_Q\left(x^k - \gamma_k\nabla_x f(x^k,\xi^k)\right),    
\end{equation}
где $\pi_Q$ -- обычное (евклидово) проектирование на множество $Q$, а $\xi^k$ выбирается независимо от $\xi^0,...,\xi^{k-1}$. 

Если $f(x)$ -- выпуклая, то, выбирая $\gamma_k\equiv \frac{R}{M\sqrt{N}}$  можно получить  
  $$\EE f(\bar{x}^N) - f(x^*) \le \frac{MR}{\sqrt{N}},$$
  где $R = \|x^0 - x^*\|_2$ (если $x_*$ не единственное, то в этой формуле можно выбирать ближайшее по $2$-норме к $x^0$), $\EE_{\xi} \|\nabla f(x,\xi)\|_2^2 \le M^2$ при $x\in Q$ (можно сузить на пересечение $Q$ с некоторым шаром с центром в $x^0$ и радиусом порядка $2R$ \cite{sadiev2023high} -- аналогичное замечание можно делать и относительно всех остальных констант, вводимых далее), $\bar x^N = \frac{1}{N}\sum_{k=0}^{N-1} x^k$.
Если $f(x)$ -- $\mu$-сильно выпуклая функция, то, выбирая $\gamma_k =  \frac{1}{\mu (k+1)}$, можно получить 
\begin{equation}\label{SC}
      \frac{\mu}{2}\EE\|\bar{x}^N - x_*\|_2^2 \le\EE f(\bar{x}^N) - f(x^*) \le \frac{M^2}{\mu N}.
\end{equation}
Приведенные оценки в общем случае (без дополнительных предположений)  не могут быть улучшены \cite{немировский1979сложность}. То есть не существует другого способа агрегирования выборки $\left\{\xi^k\right\}_{k=1}^N$, который давал бы оценки лучше (с точностью до числового множителя) приведенных. Для невыпуклой $f(x)$ гарантировать сходимость к глобальному минимуму уже нельзя. Тем не менее, на практике \eqref{eq:projSGD} и его вариации активно применяется и для невыпуклых задач, например, обучения нейронных сетей.

Отметим, что для $Q = \R^n$, видоизменив сам метод \eqref{eq:projSGD}, при некоторых дополнительных условиях результат \eqref{SC} можно уточнить следующим образом  \cite{li2022root} (приведенная оценка также будет неулучшаемой)
\begin{equation}\label{eq:Jordan}
\EE\|\bar{x}^N - x_*\|_2^2\ \lesssim \frac{\text{Tr}\left(\left[\nabla^2 f(x_*)\right]^{-1}\Sigma\left[\nabla^2 f(x_*)\right]^{-1}\right)}{N} + O\left(\frac{1}{N^{3/2}}\right),
\end{equation}
где $\Sigma = \EE_{\xi}\left[\nabla f(x_*,\xi)\nabla f(x_*,\xi)^T\right]$.
Собственно, в этом месте остановимся, чтобы поподробнее рассказать о вкладе Б.Т. Поляка в получении такого рода результатов. В 70-е годы прошлого века Б.Т. Поляк совместно с Я.З. Цыпкиным исследовали следующие псевдоградиентные процедуры стохастического агрегирования (то есть алгоритмы решения задачи \eqref{sec:stoch})
$$x^{k+1}= x^k - \gamma_k\phi(\nabla_x f(x^k,\xi^k)), $$
в которых за счет выбора вектор-функции $\phi(z)$ хотелось получить как можно лучшую скорость сходимости. Базируясь на результатах \cite{невельсон1972} в работе \cite{поляк1980оптимальные} при аддитивном шуме: $\nabla f(x,\xi) = \nabla f(x) + \xi$ удалось показать, что оптимальным будет такой выбор $\gamma_k = k^{-1}$, $\phi(z) = \left[\nabla^2 f(x_*)\right]^{-1}J^{-1}\nabla \ln p_{\xi}(z)$, где $p_{\xi}(z)$ -- функция плотности распределения случайного вектора $\xi$, а информационная матрица Фишера считается по формуле $J = \int \nabla \ln p_{\xi}(z) \left[\nabla \ln p_{\xi}(z)\right]^T  p_{\xi}(z) dz$. Под оптимальностью понимается следующее: при $N\to\infty$ имеет место центральная предельная теорема (ЦПТ) в форме:
$$\sqrt{N}\left(x^N- x_*\right) \in \mathcal{N}\left(0,\left[\nabla^2 f(x_*)\right]^{-1}J\left[\nabla^2 f(x_*)\right]^{-1}\right),$$
и при указанном выше способе выбора $\phi(z)$ ковариационная матрица наименьшая (в смысле полуопределенного отношения частичного порядка).

Однако во многих реальных приложениях плотность распределения $p_{\xi}(z)$ неизвестна. Поэтому использовать ее при выборе $\phi(z)$ нежелательно. Это приводит к корректировке оптимальной процедуры $\phi(z) = \left[\nabla^2 f(x_*)\right]^{-1} z$ и корректировке основного результата (ЦПТ): при $N\to\infty$
\begin{equation}\label{eq:CLT}   
\sqrt{N}\left(x^N- x_*\right) \in \mathcal{N}\left(0,\left[\nabla^2 f(x_*)\right]^{-1}\Sigma\left[\nabla^2 f(x_*)\right]^{-1}\right).
\end{equation}
Здесь использовалось, что $\nabla f(x_*) = 0$. Отметим, что аддитивность шума $\xi$ при этом не требуется. Этот результат также будет оптимальный в классе методов без доступа к $p_{\xi}(z)$. Однако даже в такой формулировке результат едва ли можно назвать практичным, поскольку для задания $\phi(z)$ требуется знать $\nabla^2 f(x_*)$, что возможно, в основном, только для задач квадратичной стохастической оптимизации. Ключевое наблюдение, позволяющее решить отмеченную проблему,  пришло в голову Борису Теодоровичу в конце 80-х годов во сне, и оказалось удивительным по простоте \cite{поляк1990новый,polyak1992acceleration}: $\phi(z) = z$, $\gamma_k \sim k^{-\eta}$, $\eta \in \left(1/2,1\right)$, и в качестве выхода алгоритма предлагается использовать не $x^N$, а $\bar{x}^N$. В этом случае \eqref{eq:CLT} (с заменой $x^N$ на $\bar{x}^N$) останется верным. Таким образом, было показано, как избавиться от типично недоступного предобуславливателя $\left[\nabla^2 f(x_*)\right]^{-1}$. Аналогичное можно проделать и для стохастических вариационных неравенств \cite{fort2015central}. 
Асимптотический вариант \eqref{eq:Jordan} очевидным образом получается из \eqref{eq:CLT}. Получение неасимптотического варианта требует больших усилий (на появление первых таких результатов ушло еще более 20 лет \cite{bach2016highly}).

Отметим, что близкая идея (однако реализованная в существенно меньшей общности) использования $\bar{x}^N$ вместо $x^N$  независимо была приблизительно в то же время предложена и на западе Д.~Руппертом \cite{ruppert1988efficient}.

Для ряда постановок задач, например, когда множество $Q$ является симплексом, выгоднее (с точки зрения того, как в итоговую оценку будет входить размерность $n$ посредством $M$ и $R$) использовать неевклидово проектирование (в частности, для симплекса лучше использовать проектирование согласно дивергенции Кульбака--Ляйблера, которое приводит к экспоненциальному взвешиванию  компонент стохастического градиента). Соответствующие обобщения SGD принято назвать стохастический метод зеркального спуска (stochastic mirror descent -- SMD) \cite{немировский1979сложность,nemirovski2009robust}. Проблему неадаптивного  выбора шага $\gamma_k$ (требуется заранее знать $N$) в выпуклом случае решает вариация SMD -- стохастический метод двойственных усреднений (stochastic dual averaging method) \cite{nesterov2009primal}. Однако более изящно проблема выбора шага решается в AdaGrad версии SGD \cite{duchi2011adaptive}, в которой 
$$\gamma_k = \frac{R}{\sqrt{\sum_{j=1}^k\|\nabla f(x^j,\xi^j)\|_2^2}}.$$
При таком выборе шага не требуется и знание глобальной константы $M$. В современных работах избавляются также и от зависимости $R$ в шаге (класс Parameter-free SGD, к которому относятся, например, DoG \cite{ivgi2023dog} и конструкция Mechanic \cite{cutkosky2023mechanic}). К сожалению, в общем  сильно выпуклом случае пока неизвестно, как можно было бы избавиться от необходимости знания $\mu$ (продвижения имеются лишь в частных случаях, например, когда $f(x^*)$ известно).  Напомним, что аналогичная проблема была и для ускоренных детерминированных методов, см. конец раздела \ref{sec:accelerated}.

В случае, если дополнительно известно, что функция $f(x)$ -- гладкая (имеет Липшицев градиент), то SGD можно существенно ускорить за счет батч-параллелизации (замены стохастического градиента на выборочное среднее стохастических градиентов на независимых реализациях): $$\nabla f(x,\xi) \to \frac{1}{b}\sum_{i=1}^b \nabla f(x,\xi^i),$$ где $\xi^i$ -- независимые одинаково распределенные, как $\xi$, а $b\ge 1$ -- размер батча, который можно вычислять параллельно.  Действительно, рассмотрим, следуя Б.Т. Поляку \cite{поляк1983введение}, более точную оценку скорости сходимости  SGD в гладком случае (в \cite{поляк1983введение} используется глобальная оценка дисперсии $\sigma^2$, однако, несложно показать, что достаточно использовать введенную далее дисперсию в решении $\sigma_*^2$, см., например, \cite{stich2019unified}):
\begin{equation}\label{eq:param}
\EE\left[\|x^{N} - x_*\|_2^2\right] \le \|x^0 - x_*\|_2^2 \left(1 - \gamma\mu\right)^N + \frac{2\gamma\sigma_*^2}{\mu},
\end{equation}
где $\sigma_*^2 =\EE_{\xi}\left[\|\nabla f(x_*,\xi) - \nabla f(x_*)\|_2^2\right]$,
  $\|\nabla f(y,\xi) - \nabla f(x,\xi)\|_2 \le L\|y-x\|_2$, $\gamma_k \equiv \gamma \le 1/(2L)$. 
  Несложно также показать, что при выборе $\gamma \equiv 1/(2L)$ за счет батчирования ($\sigma_*^2 \to \sigma_*^2/b$) можно выравнять оба слагаемых в правой части \eqref{eq:param}, и получить такую версию \eqref{eq:non_accel_mu} (тут $N$ --  число вычислений $\nabla f(x,\xi)$):
\begin{equation*}
    \EE \|x^N - x_*\|_2^2 \lesssim R^2\exp\left(-\frac{\mu}{2L}N\right) + \frac{\sigma_*^2}{\mu^2 N}.
\end{equation*}
При этом для ожидаемой невязки по функции можно получить такую оценку:
\begin{equation}\label{eq:non_accel_mu_stoch}
    \EE f(x^N) - f(x_*) \lesssim LR^2\exp\left(-\frac{\mu}{2L}N\right) + \frac{\sigma_*^2}{\mu N}.
\end{equation}
 Отметим, что в последнее время в связи с обучением нейронных сетей огромных размеров возникает потребность в изучении роли перепараметризации, что  можно сформулировать как малость дисперсии $\sigma_*^2$. Современное состояние развития этого направления для неускоренных стохастических градиентных методов описано, например, здесь \cite{gorbunov2023unified}. Малость $\sigma_*^2$ -- означает линейную скорость сходимости в небольшую окрестность решения. Такую картину, наверняка, многие наблюдали, на практике, решая задачи обучения. А именно, если выбирать шаг $\gamma$ (learning rate) достаточно большим, то в ряде случаев можно наблюдать линейную скорость сходимости SGD. Но чем больше шаг $\gamma$, тем больше окрестность, внутри которой метод перестает сходиться. Для дальнейшего продвижения требуется уменьшение шага или батчирование.

 Многое из того, что написано выше без каких-то существенных изменений переносится и на стохастические вариационные неравенства (седловые задачи), см., например, обзор \cite{beznosikov2023smooth}, написание которого было инициировано Б.Т. Поляком летом 2022 года. Насколько нам известно, этот обзор, по-видимому, является последней научной работой Бориса Теодоровича.

\subsection{Ускоренные версии стохастического градиентного спуска}\label{sec:accel_stoch}
Прежде всего заметим, что \eqref{eq:non_accel_mu_stoch} в форме 
$$ \EE f(x^N) - f(x_*) \lesssim LR^2\exp\left(-\frac{\mu}{2L}N\right) + \frac{\sigma^2}{\mu N},$$
где $$\EE_{\xi}\left[\|\nabla f(x,\xi) - \nabla f(x)\|_2^2\right]\le\sigma^2$$ справедливо при более слабом предположении о Липшицевости градиента 
$$\|\nabla f(y) - \nabla f(x)\|_2 \le L\|y-x\|_2.$$ В таких же условиях можно улучшить (ускорить) оценку \eqref{eq:non_accel_mu_stoch} если за основу брать ускоренный детерминированный метод и заменять в нем градиент на должным образом пробатченный стохастический градиент, см., например, \cite{lan2012optimal,гасников2018универсальный,lan2020first}, простое изложение имеется в \cite{гасников2018современные,gasnikov2022power} (следует сравнить с \eqref{eq:accel_mu}):
\begin{equation}\label{eq:accel_stoch_mu}
 \EE f(x^N) - f(x_*) \lesssim LR^2\exp\left(-\sqrt{\frac{\mu}{4L}}N\right) + \frac{\sigma^2}{\mu N}.
 \end{equation}
Аналогично для выпуклого случая (следует сравнить с \eqref{eq:accel}):
\begin{equation}\label{eq:accel_stoch}
     \EE f(x^N) - f(x_*) \lesssim \frac{LR^2}{N^2} + \frac{\sigma^2R^2}{\sqrt{N}}.
\end{equation}
Если дополнительно известно, что $$\|\nabla f(y,\xi) - \nabla f(x,\xi)\|_2 \le L\|y-x\|_2,$$ то приведенные оценки можно уточнить следующим образом  \cite{woodworth2021even,ilandarideva2023accelerated} (здесь, как и раньше, $b$ -- размер батча, только сейчас мы явно его прописываем, поскольку батчированию тут поддаются слагаемые не только содержащие $\sigma_*^2$):
$$\EE f(x^N) - f(x_*) \lesssim LR^2\exp\left(-\sqrt{\frac{\mu}{4L}}N\right) + LR^2\exp\left(-\frac{\mu}{2L}bN\right)+ \frac{\sigma_*^2}{\mu b N},$$
$$\EE f(x^N) - f(x_*) \lesssim \frac{LR^2}{N^2} + \frac{LR^2}{bN}  + \frac{\sigma_*^2R^2}{\sqrt{bN}}.$$
Причем все приведенные выше оценки в разделе~\ref{sec:accel_stoch} имеют место и для задач с ограничениями простой структуры и для неевклидова проектирования. Более того, все эти оценки -- оптимальны, то есть не могут быть в общем случае улучшены (с точностью до числовых множителей, в том числе в показатели экспоненты).

Также как и для обычного  SGD на практике большую роль играет адаптивность метода. Добавление различных вариантов моментного ускорения к адаптивным методам (типа AdaGrad), упомянутым в разделе~\ref{sec:SGD}, порождает популярную линейку современных методов типа Adam, AdamW, RMSProp, AdaDelta и т.д., активно использующихся для обучения нейронных сетей. Для выпуклых постановок задач имеется и теоретическое обоснование \cite{kavis2019unixgrad,ene2021adaptive}. Однако вопрос о создании полностью адаптивного ускоренного метода для решения задач выпуклой стохастической оптимизации, насколько нам известно, пока окончательно еще не решен.

В предположении $Q = \R^n$ отметим концепцию мультипликативных помех, развиваемую в работах Б.Т. Поляка в 70-80-е годы прошлого века \cite{поляк1983введение,немировский1985оптимальные}. На современный манер условие, которому удовлетворяют введенные помехи, можно было бы называть условием сильного роста (strong growth):
\begin{equation}
\label{eq:sg}
\EE \|\nabla f(x,\xi)\|_2^2 \le \rho_{sg} \|\nabla f(x)\|_2^2 + \sigma_{sg}^2, \quad \rho_{sg}, \sigma_{sg} \geq 0.
\end{equation}
Такому условию в гладком случае, например, удовлетворяет координатные методы \cite{nesterov2012efficiency}, где рандомизация в стохастическом градиенте возникает за счет случайного выбора координаты, по которой считается частная производная вместо вычисления полного градиента, при этом можно выбирать сразу несколько координат (батч) и сэмплировать не обязательно равномерно, а исходя из свойств производных по направлению \cite{richtarik2016optimal, qu2016coordinate}. Также под неравенство \eqref{eq:sg} подходят градиенты, к которым применяется оператор сжатия \cite{alistarh2017qsgd}. Такого рода рандомизация используется в распределенной оптимизации для передачи меньшего числа информации. К примерам операторов сжатия относятся и уже упомянутый выше случайный выбор координат, различные рандомизированные квантизации и округления \cite{beznosikov2020biased}. 

Для неускоренных методов, использующих стохастический градиент вида \eqref{eq:sg}, начало построения теории было заложено в уже упомянутых работах \cite{поляк1983введение,немировский1985оптимальные}. В связи с активным развитием машинного обучения стохастические методы оптимизации стали широко исследоваться в сообществе, в частности, было переоткрыто и предположение сильного роста \cite{schmidt2013fast}. На данный момент для неускоренных методов, например, для классического SGD вида \eqref{eq:projSGD} имеется хорошо разработанная теория сходимости -- см., например, обзорную работу \cite{gorbunov2023unified}. В частности, для $L$-гладкой выпуклой целевой функции $f$ справедлива следующая оценка скорости сходимости после $N$ итераций SGD:
 \begin{equation*}
     \EE f(\bar x^N) - f(x_*) \lesssim \frac{\rho_{sg} L \|x^0 - x_*\|^2_2}{N} + \frac{\sigma_{sg} \|x^0 - x_*\|_2}{\sqrt{N}},
\end{equation*}
где $\bar x^N = \frac{1}{N} \sum_{k=0}^{N-1} x^k$.
Если функция является не просто выпуклой, а $\mu$-сильно выпуклой, то можно улучшить оценку и получить, что
\begin{equation*}
 \EE \|x^N - x_*\|_2^2 \lesssim \exp\left(-\frac{\mu N}{\rho_{sg} L}\right) \|x^0 - x_*\|_2^2 + \frac{\sigma^2_{sg}}{\mu^2 N}.
 \end{equation*}
Для ускоренных вариантов теория немного беднее, но основные результаты уже были получены \cite{vaswani2019fast}. Отметим, что классический ускоренный метод \cite{нестеров1983метод} не подходит для такой постановки \eqref{eq:sg} и необходимо использовать дополнительный моментный член (momentum term) \cite{nesterov2012efficiency,vaswani2019fast}. Тогда в предположении о $L$-гладкости и выпуклости функции $f$ можно получить следующую оценку скорости сходимости:
\begin{equation*}
     \EE f(\bar x^N) - f(x_*) \lesssim \frac{\rho^2_{sg} L \|x^0 - x_*\|^2_2 }{N^2} + \frac{\sigma_{sg} \|x^0 - x_*\|_2}{\sqrt{N}},
\end{equation*}
а для $\mu$-сильно выпуклой функции:
\begin{equation*}
 \EE \| x^N - x_*\|^2_2 \lesssim \exp\left(-\sqrt{\frac{\mu N^2}{4 \rho^2_{sg} L}}\right) \|x^0 - x_*\|_2^2 + \frac{\sigma_{sg}^2}{\mu^2 N}.
 \end{equation*}
Отметим, что приведенные результаты удалось с некоторыми оговорками и ослаблением перенести на марковский шум \cite{beznosikov2023first}.
% , что позволило найти им довольно неожиданное приложение в гладкой децентрализованной оптимизации на меняющихся графах \cite{metelev2023decentralized}, как раз для получения ускорения по числу обусловленности коммуникационного графа  при некоторых дополнительных предположениях (см. также конец раздела \ref{sec:accelerated}).

Между тем, легко заметить, что предположение \eqref{eq:sg} можно релаксировать до условия слабого роста (weak growth):
\begin{equation}
    \label{eq:wg}
    \EE \|\nabla f(x,\xi)\|_2^2 \le \rho_{wg} (f(x) - f(x_*)) + \sigma_{wg}^2 \quad \rho_{wg}, \sigma_{wg} \geq 0.
\end{equation}
Если выполнено \eqref{eq:sg}, то для выпуклой и $L$-гладкой функции $\rho_{wg} = 2L \rho_{sg}$ и $\sigma_{wg} = \sigma_{sg}$. Условие \eqref{eq:wg} является не менее распространенным. В частности, одним из популярных примеров применимости \eqref{eq:wg} является гладкость в среднем, а именно, нам необходимо предположить, что для любой реализации $\xi$ функция $f(\cdot,\xi)$ является $L(\xi)$-гладкой и выпуклой, и отсюда получить:
\begin{equation}
    \label{eq:expected_smoothness}
    \begin{split}
        \EE \|\nabla f(x,\xi)\|_2^2 \le& 2 \EE\|\nabla f(x,\xi) - \nabla f(x_*,\xi) \|_2^2 + 2 \EE\|\nabla f(x_*,\xi)\|_2^2 
    \\
    \le& 2 \mathcal{L}^2 (f(x) - f(x_*) ) + 2 \EE\|\nabla f(x_*,\xi)\|_2^2,
    \end{split}
\end{equation}
где $\mathcal{L}^2 = \EE[L^2(\xi)]$. Отметим, что константа $\mathcal{L}$ может быть значительно хуже, чем $L$ -- константа гладкости (Липшицевости градиента) $f$. Выкладка \eqref{eq:expected_smoothness} является самым популярным в литературе примером предположения \eqref{eq:wg}.
В частности, оно появляется в работе \cite{Moulines_2011}, где авторы прежде всего мотивируются классической задачей наименьших квадратов. В дальнейшем исследование предположений \eqref{eq:wg} и \eqref{eq:expected_smoothness} было обобщено на неравномерную рандомизацию, которая учитывает свойства батчей. В работе \cite{gower2019sgd} предлагается довольно исчерпывающая теория для рандомизации вида \eqref{eq:expected_smoothness} с разбором большого числа частных случаев. А именно, классический SGD \eqref{eq:projSGD} для выпуклой целевой функции $f$ имеет следующие гарантии сходимости:
 \begin{equation*}
     \EE f(\bar x^N) - f(x_*) \lesssim \frac{\rho_{wg} \|x^0 - x_*\|^2_2}{N} + \frac{\sigma_{wg} \|x^0 - x_*\|_2}{\sqrt{N}},
\end{equation*}
Если функция является дополнительно $\mu$-сильно выпуклой, то можно получить, что
\begin{equation*}
 \EE \|x^N - x_*\|_2^2 \lesssim \exp\left(-\frac{2\mu N}{\rho_{wg}}\right) \|x^0 - x_*\|_2^2 + \frac{\sigma^2_{sg}}{\mu^2 N}.
 \end{equation*}
Говоря о предположениях \eqref{eq:sg} и \eqref{eq:wg}, важно заметить, что для многих частных случаев $\sigma_{sg}$ и $\sigma_{wg}$ равны 0, а это может значительно улучшить гарантии сходимости. Здесь можно отметить популярные и довольно часто встречающиеся примеры перепараметризации ($\nabla f(x_*,\xi) = 0$ для всех $\xi$) \cite{vaswani2019fast} и интерполяции ($f(x, \xi) \geq 0$ и $f(x_*,\xi) = 0$ для всех $x$ и $\xi$) \cite{ma2018power}. Также для упомянутых выше координатных методов справедливо, что $\sigma_{sg} = 0$. Но для самых простых методов с сжатием $\sigma_{sg} \neq 0$. Это мотивировало сообщество создать более продвинутые методы, использующие компрессию \cite{mishchenko2019distributed}. Но стохастический градиент в данных подходах не получается описать с помощью \eqref{eq:sg} и \eqref{eq:wg}. Можно ввести более сложное предположение \cite{gorbunov2020unified}:
\begin{equation*}
    \begin{split}
        \EE \left[ \|\nabla f(x^k,\xi^k)\|_2^2 \mid x^k \right] \le 2\rho_{\sigma_k, 1} (f(x^k) - f(x_*)) + \rho_{\sigma_k, 2} \sigma^2_k + \sigma^2_{\sigma_k, 1},
        \\
    \EE \left[\|\sigma^2_{k+1}\|_2^2 \mid x^k \right] \le (1 - p) \sigma^2_{k} + 2 (f(x^k) - f(x_*)) + \sigma^2_{\sigma_k, 2}.
    \end{split}
\end{equation*}
С помощью него можно унифицировано анализировать не только многие современные методы с сжатием, но и популярные алгоритмы, использующие технику редукции дисперсии \cite{defazio2014saga,johnson2013accelerating,kovalev2020don}, а также продвинутые координатные методы \cite{hanzely2018sega}. Важной деталью данного предположения является наличие вспомогательной последовательности $\{\sigma^2_k\}$, которая является уникальной для каждого метода. Эта последовательность обладает важным свойством сходимости, которое и позволяет рассмотреть функцию Ляпунова, состоящую из двух частей: классической вида $\| x^k - x_*\|_2^2$ или $f(x^k) - f(x_*)$ и дополнительной, завязанной на $\sigma^2_k$. Например, для $\mu$-сильно выпуклой функции метод SGD \eqref{eq:projSGD} с постоянным шагом $\gamma_k = \gamma$ таким, что
$$
\gamma \leq \min \left\{ \frac{1}{\mu}; \frac{1}{\rho_{\sigma_k, 1} + \frac{2\rho_{\sigma_k, 3} \rho_{\sigma_k, 2}}{p} } \right\},
$$
может гарантировать следующую оценку сходимости \cite{gorbunov2020unified}:
\begin{equation*}
 \EE V_N \lesssim \exp\left(- \min\left\{ \gamma \mu; \frac{p}{2}\right\} N\right) V_0 + \frac{\left(\sigma^2_{\sigma_k, 1} + \frac{\rho_{\sigma_k, 2} \sigma^2_{\sigma_k, 2}}{p}\right)\gamma^2}{\min\{ \gamma \mu ; p\}},
 \end{equation*}
где $V_k = \|x^k - x_* \|^2_2 + \frac{\gamma^2 \rho_{\sigma_k, 2} \sigma_k^2}{2p}$.
Похожие результаты имеются и для выпуклой целевой функции $f$ \cite{khaled2023unified}, а также для стохастических вариационных неравенств и седловых задач \cite{beznosikov2023stochastic, beznosikov2023unified}. Насколько нам известно, на данный момент в условиях слабого роста не известно можно ли добиться ускорения, и если можно, то каким образом?

Наряду с предположениями \eqref{eq:sg} и \eqref{eq:wg}, можно рассмотреть и похожее условие вида 
\begin{equation*}
    \EE \|\nabla f(x,\xi)\|_2^2 \le \rho_{x}  \| x - x_*\|^2_2 + \sigma_{x}^2.
\end{equation*}
Касательно него можно выделить следующие работы \cite{hsieh2020explore,gorbunov2022stochastic,beznosikov2023first,metelev2023decentralized}.

Все приведенные выше результаты формулировались в терминах сходимости по математическому ожиданию. Для такой сходимости было достаточно ограниченности второго момента стохастического градиента. В действительности, за счет клиппирования на базе описанных методов можно строить робастные версии, которые гарантированно сходятся с такой же скоростью, но уже в терминах вероятностей больших отклонений, причем имеет место почти субгауссовская концентрация \cite{назин2019алгоритмы,sadiev2023high,gorbunov2023highprobability}. Заметим, что идея клиппирования (нормализации градиента), как способ борьбы с тяжелыми хвостами, в скалярном случае $n=1$, по-видимому, впервые появилась в 1973 году в работе Б.Т. Поляка и Я.З. Цыпкина \cite{поляк1973оптимальные} как частный случай того, как можно выбирать функцию $\phi(z)=\min\left\{1,\frac{\lambda}{\|z\|_2}\right\}z$, в псевдоградиентной процедуре из раздела~\ref{sec:SGD}. Отметим, что результаты Поляка--Цыпкина, кратко описанные в разделе~\ref{sec:SGD}, недавно были перенесены как раз на постановки задач, в которых аддитивный шум $\xi$ имеет тяжелые хвосты распределения, в том числе не предполагающие наличие конечной дисперсии у шума \cite{jakovetic2023nonlinear}.

В заключение раздела отметим, что недавно ускоренные версии тензорных методов типа Нестерова--Поляка были распространены на достаточно гладкие задачи стохастической оптимизации. В частности, для методов второго порядка полученные  результаты оптимальны по числу вызовов стохастических градиентов и числу вызовов стохастических гессианов \cite{agafonov2023advancing}. 

\subsection{Безградиентные методы}\label{sec:GF}
Частным случаем стохастики $\xi$ в $\nabla f(x,\xi)$ может быть рандомизация, которая не <<дана извне>>, а привнесена нами самими. Введение в метод рандомизации может иметь разные причины. Например, ярко об этом написано в статье Ю.Е. Нестерова про покомпонентные методы \cite{nesterov2012efficiency} или в фундаментальной статье  А.С. Немировского и др. \cite{nemirovski2009robust} в части рандомизации умножения матрицы на вектор из единичного симплекса. Но, пожалуй, самый известный пример рандомизации в стохастической оптимизации -- это рандомизация суммы: для целевого функционала вида взвешенной суммы в качестве стохастического градиента используется случайно выбранное слагаемое, см., например, \cite{гасников2018современные}. Однако в этом разделе будут описаны, так называемые, безградиентные (поисковые) методы или методы нулевого порядка, в которых рандомизация -- это вынужденная мера, связанная с отсутствием необходимой информации. Такие методы периодически встречались в работах Бориса Теодоровича, см., например, 
\cite{поляк1983введение,граничин2003рандомизированные}, и одна из предложенных им конструкций, которая в последнее время вызывает определенный интерес, будет далее изложена.

% \ag{\textbf{Александру Лобанову} Изложить суть ядерной оценки производной по направлению
% \cite{поляк1990оптимальные,bach2016highly,gasnikov2022randomized,akhavan2023gradient}. Причем тут стоит отметить и нижние оценки и известные сейчас верхние оценки, в том числе включающие условие Поляка--Лоясиевича ну и наши всевозможные уточнения и обобщения на случай разных концепций шума, разных режимов, например, \cite{lobanov2023zero,pasechnyuk2023upper} ... Но можно ограничиться только гладкими задачами (в том числе, повышенно гладкими). Это сократит объем материала. А на негладкие просто можно будет сослаться \cite{gasnikov2022power,gasnikov2022randomized,kornilov2023gradient}. Причем, по возможности, тут даже лучше ссылаться не на \cite{kornilov2023gradient}, а на ее правильную версию с ускорением, которую подали на НИПС2023.}

Прежде всего, рассмотрим выпуклую задачу оптимизации:
\begin{equation}\label{eq:black_box_problem}
    \min_{x \in Q} f(x),
\end{equation}
которая существенно отличается от предыдущих постановок задач, в частности от \eqref{SO}, тем, что оракул может выдать только значение целевой функции $f(x)$ в запрошенной точке $x$.
%, возможно, с некоторым ограниченным враждебным шумом\footnote{В литературе встречаются две концепции враждебного шума: детерминированный шум $\delta(x)$ (представлен на этой странице) и стохастический шум $\xi$ (представлен в Алгоритме \ref{GradientFree}).} $|\delta(x)| \leq \Delta$
Такой оракул часто упоминается в литературе как оракул нулевого порядка или безградиентный оракул \cite{Rosenbrock_1960}. 
%:$f_\delta(x) = f(x) + \delta(x)$
Из-за невозможности получить информацию о $l$-ой производной функции (например, градиент функции $f$) для решения задачи \eqref{eq:black_box_problem} зачастую прибегают к помощи  численных методов нулевого порядка, которые основываются на методах первого порядка, заменяя истинный градиент на различные модели аппроксимации градиента \cite{Kiefer_1952}. Например, когда функция $f(x)$ является не просто гладкой, а имеет повышенную гладкость, т.е. функция $f: \mathbb{R}^d \rightarrow \mathbb{R}$ имеет непрерывные частные производные до $l$-го порядка включительно и для всех $x,z \in Q$ удовлетворяет условию Гельдера:
\begin{equation*}
    \left| f(z) - \sum_{0 \leq |n| \leq l} \frac{1}{n!} D^n f(x) (z-x)^n \right| \leq L_\beta \| z - x \|^\beta_2,
\end{equation*}
где $l < \beta$, $L_\beta>0$, $n~=~(n_1, ..., n_d)$ -- мультиидекс, $n_i\geq 0$ -- целые, $n!~=~n_1! \cdots n_d!$, $|n| = n_1 + \cdots + n_d$, и $\forall v = (v_1, ..., v_d) \in \mathbb{R}^d$, а также $D^n f(x) v^n = \frac{\partial ^{|n|} f(x)}{\partial^{n_1}x_1 \cdots \partial^{n_d}x_d} v_1^{n_1} \cdots v_d^{n_d}$, то при создании безградиентного алгоритма важно подобрать такую аппроксимацию градиента, которая будет использовать преимущества повышенной гладкости функции ($\beta \geq 2$, где $\beta$ -- порядок~гладкости~функции~$f$). Такую оценку производной по направлению предложили в 1990 году Б.Т. Поляк и А.Б. Цыбаков \cite{поляк1990оптимальные}, которая в дальнейшем стала называться <<ядерная>> аппроксимация и активно использоваться \cite{bach2016highly, gasnikov2022randomized, Lobanov_2023,  akhavan2023gradient}:
\begin{equation}\label{kernel_approx}
    \Tilde{\nabla}f(x,\mathbf{e}) = d \frac{f(x+\tau r \mathbf{e}) - f(x - \tau r \mathbf{e})}{2 \tau} K(r) \mathbf{e},
\end{equation}
где $\tau > 0$, $\mathbf{e}$ -- равномерно распределенный на $S_2^d(1):=\left\{ x \in \mathbb{R^d} : \| x \|_2 = 1 \right\}$, $r$~--~равномерно распределенный на отрезке $[-1,1]$, $\mathbf{e}$ и $r$ независимы, $K:~[-1,1]~\rightarrow~\mathbb{R}$ -- фиксированная функция (ядро), которая удовлетворяет следующим условиям:
\begin{gather*}
    \mathbb{E}[K(u)] = 0, \;\;\; \mathbb{E}[u K(u)] = 1, \;\;\;
    \mathbb{E}[u^j K(u)] = 0, \;\;\; j=2,...,l, \;\;\; \mathbb{E}[|u|^\beta |K(u)|] < \infty.
\end{gather*}
Одним из основных достоинств этой аппроксимации градиента является тот факт, что ядерная аппроксимация \eqref{kernel_approx} требует всего два вычисления значения (реализации) функции на итерации, поскольку информация о повышенной гладкости учитывается в <<ядре>>. Этот факт существенно улучшает оракульную сложность алгоритма, который использует конечно-разностную схему более высокого порядка в качестве оценки градиента \cite{Berahas_2022}, поскольку данная аппроксимация требует большего числа вызовов безградиентного оракула на каждой итерации. К 2020 году появились интересные результаты о скорости сходимости для безградиентного алгоритма \cite{поляк1990оптимальные,bach2016highly, Akhavan_2020,Novitskii_2021}: Стохастический метод проекции градиента нулевого порядка, описание которого можно найти в Алгоритме \ref{GradientFree}.

%\floatname{algorithm}{Алгоритм}
\begin{algorithm}
\caption{Стохастический метод проекции градиента нулевого порядка}\label{GradientFree}
\begin{algorithmic}[1]

\STATE \textbf{Requires}: Ядро $K: [-1,1] \rightarrow \mathbb{R}$, размер шага $\eta_k$, сглаживающий параметр $\tau_k$.

\STATE \textbf{Initialization}: Сгенерировать скалярный числа $r_1, ..., r_N$ равномерно распределенные на отрезке $[-1,1]$ и вектора $\textbf{e}_1, ..., \textbf{e}_N$ равномерно распределенные на единичной Евклидовой сфере $S_2^d(1)$.
\FOR{$k=1, \, \dots, \, N$}
\STATE $f_{\xi_k} := f(x_k + \tau_k r_k \textbf{e}_k) + \xi_k$, \quad  $f_{\xi'_k} := f(x_k - \tau_k r_k \textbf{e}_k) + \xi'_k$

\STATE $\Tilde{\nabla}f(x_k,\textbf{e}_k) := \frac{d}{2 \tau_k} \left( f_{\xi_k} - f_{\xi'_k}\right) \textbf{e}_k K(r_k)$

\STATE $x_{k+1} := \text{Proj}_{Q} \left( x_k - \gamma_k \Tilde{\nabla}f(x_k,\textbf{e}_k) \right)$
\ENDFOR
\RETURN $\left\{ x_k \right\}_{k=1}^{N}$.
\end{algorithmic}
\end{algorithm}

Как видно из строчки 4 Алгоритма \ref{GradientFree} $f_{\xi}$ выступает в роли безградиентного оракула, где $\xi \neq \xi' $ -- стохастический шум, который характеризует конкретную реализацию (т.е. $f_{\xi}$ -- это значение целевой функции на реализации $\xi$). Именно поэтому строчку 5 называют аппроксимацией градиента с одноточечной обратной связью. Данная концепция стохастического шума \cite{гасников2016стохастические, Akhavan_2020, Novitskii_2021, Граничин_2023} формально определяется следующим образом: $\mathbb{E}[\xi^2] \leq \Tilde{\Delta}^2$ и $\mathbb{E}[\xi'^2] \leq \Tilde{\Delta}^2$, $\Tilde{\Delta} \geq 0$, а случайные величины $\xi$ и $\xi'$ не зависят от $\textbf{e}$ и $r$. Более того, эта концепция шума не требует предположение о нулевом среднем $\xi$ и $\xi'$, поскольку достаточно того, что $\mathbb{E}[\xi \textbf{e}] = 0$ и $\mathbb{E}[\xi' \textbf{e}] = 0$. В Таблице \ref{table:GF_results} представлены результаты работ \cite{bach2016highly,Akhavan_2020, Novitskii_2021} через зависимости $N(\varepsilon)$ для различных предположений о выпуклости функции (выпуклая/сильно выпуклая функция), где $N$ -- число последовательных итераций, совпадающее (с точность до константы) с общим числом обращений к оракулу нулевого порядка $T = 2N$. Все оценки Таблицы \ref{table:GF_results} соответствуют случаю, когда~$\Tilde{\Delta}$~не~мало.

\begin{table}[ht] 

    \caption{\label{table:GF_results} Зависимость числа итераций $N$ от желаемой точности задачи $\varepsilon$, размерности $d$, константы сильной выпуклости $\mu$ и порядка гладкости функции $\beta$}
    \resizebox{\linewidth}{!}{
    \begin{tabular}{|l|c|c|}
    \toprule
     &  Сильно выпуклый случай & Выпуклый случай \\
    \midrule
    Нижние оценки (2020) \cite{поляк1990оптимальные,Akhavan_2020}& $\Omega \left(\min\left\{ \frac{d^{1 + \frac{1}{\beta - 1}} L_\beta^{\frac{2}{\beta - 1}} \Tilde{\Delta}^2}{\left(\mu \varepsilon\right)^{\frac{\beta}{\beta - 1}}}, \frac{d^2 R^2 \Tilde{\Delta}^2}{\varepsilon^2} \right\}\right)$ & $\Omega \left(\min\left\{ \frac{d^{1 + \frac{1}{\beta - 1}} L_\beta^{\frac{2}{\beta - 1}} R^{\frac{2 \beta}{\beta - 1}} \Tilde{\Delta}^2}{\varepsilon^{2 + \frac{2}{\beta - 1}}}, \frac{d^2 R^2 \Tilde{\Delta}^2}{\varepsilon^2} \right\}\right)$ \\
    \midrule
    Новицкий и др. (2020) \cite{Novitskii_2021} & $\mathcal{\Tilde{O}} \left( \frac{d^{2 + \frac{1}{\beta - 1}} L_\beta^{\frac{2}{\beta - 1}} \Tilde{\Delta}^2}{\left(\mu \varepsilon\right)^{\frac{\beta}{\beta - 1}}} \right)$ & $\mathcal{\Tilde{O}} \left( \frac{d^{2 + \frac{1}{\beta - 1}} L_\beta^{\frac{2}{\beta - 1}} R^{\frac{2 \beta}{\beta - 1}} \Tilde{\Delta}^2}{\varepsilon^{2 + \frac{2}{\beta - 1}}} \right)$ \\
    Akhavan et al. (2020)\cite{Akhavan_2020} & $\mathcal{\Tilde{O}} \left( \frac{d^{2 + \frac{2}{\beta - 1}}  L_\beta^{\frac{2}{\beta - 1}} \Tilde{\Delta}^2}{\left(\mu \varepsilon\right)^{\frac{\beta}{\beta - 1}}} \right)$ & $\mathcal{\Tilde{O}} \left( \frac{d^{2 + \frac{2}{\beta - 1}}  L_\beta^{\frac{2}{\beta - 1}} R^{\frac{2 \beta}{\beta - 1}} \Tilde{\Delta}^2}{\varepsilon^{2 + \frac{2}{\beta - 1}}} \right)$ \\
    Bach et al. (2016) \cite{bach2016highly} & $\mathcal{O} \left( \frac{d^{2 + \frac{2}{\beta - 1}}  L_\beta^{\frac{2 \beta}{\beta - 1}} \Tilde{\Delta}^{\frac{2(\beta+1)}{\beta - 1}}}{\left(\mu \varepsilon\right)^{\frac{\beta + 1}{\beta - 1}}} \right)$ & $\mathcal{O} \left( \frac{d^{2 + \frac{2}{\beta - 1}} \left(L_\beta R \Tilde{\Delta}^2\right)^{\frac{2\beta}{\beta- 1}}}{\varepsilon^{2 + \frac{2}{\beta - 1}}} \right)$\\
    \bottomrule
    \end{tabular}
    }
\end{table}
После некоторой <<паузы>> в 2023 году авторам работы \cite{akhavan2023gradient} удалось улучшить верхнюю оценку для сильно выпуклого случая за счет более качественного анализа оценки смещения ядерной аппроксимации \eqref{kernel_approx}, учитывая, что $\kappa_\beta = \int |u|^\beta |K(u)| du$:
\begin{equation*}
    \left\| \mathbb{E} \left[\Tilde{\nabla}f(x, \textbf{e})\right] - \nabla f(x) \right\|_2 \leq \kappa_\beta \frac{L}{(l-1)!} \cdot \frac{d}{d + \beta - 1} \tau^{\beta - 1},
\end{equation*}
а также оценки второго момента ядерной аппроксимации \eqref{kernel_approx} с $\kappa = \int K^2(u) du$:
\begin{equation*}
    \mathbb{E}\left[ \| \Tilde{\nabla} f(x, \textbf{e}) \|_2^2 \right] \leq 4d \mathbb{E}\left[ \| \nabla f(x) \|_2^2 \right] + 4 d \kappa L^2 \tau^2 + \frac{\kappa d^2 \Tilde{\Delta}^2}{\tau^2}.
\end{equation*}
Основное преимущество данных оценок состоит в том, что смещение больше не зависит от размерности $d$ асимптотически. Благодаря этому в работе \cite{akhavan2023gradient} предоставили следующую верхнюю оценку итерационной сложности (совпадает с оракульной сложностью) для сильно выпуклого случая, размерность в которой не зависит от порядка гладкости:
\begin{equation*}
    N = \mathcal{O} \left( \frac{d^2  L_\beta^{\frac{2}{\beta - 1}} \Tilde{\Delta}^2}{\left(\mu \varepsilon\right)^{\frac{\beta}{\beta - 1}}} \right).
\end{equation*}

Нетрудно заметить, что в этих работах идет <<борьба>> за 
оптимальную оракульную сложность $T = 2N$. Однако рассматривая безградиентный алгоритм в последнее время авторы уделяют внимание и другим критериям оптимальности \cite{gasnikov2022randomized}, а именно оракульная сложность $T$, число последовательных итераций $N$ и максимально допустимый уровень враждебного шума, при котором всё ещё удается достичь желаемой точности $\varepsilon$. Один из способов улучшения оценок числа последовательных итераций для безградиентного алгоритма -- это взять за базу ускоренный алгоритм первого порядка (например, см. \cite{lan2012optimal, гасников2018универсальный}) и применить технику батчирования (где $B$ -- размер батча), тем самым достигнув оптимальной оценки в итерационной сложности при $B \geq 4 \kappa d$ (см. работу \cite{Lobanov_JOTA}):
$N \sim \mathcal{O} \left( \varepsilon^{-1/2} \right)$ и получив (<<конкурирующие>> результаты с Таблицей \ref{table:GF_results})  общее число обращений к оракулу в выпуклом случае для любого размера батча $B$ с $\rho_B = \max \left\{ 1, \frac{4 \kappa d}{B} \right\}$:
\begin{equation*}
    T = N \cdot B = \max \left\{ \mathcal{O}\left( \sqrt{\frac{\rho_B^2 L R^2}{ \varepsilon}} B \right), \mathcal{O}\left( \frac{d^2  L_\beta^{\frac{2}{\beta - 1}} \Tilde{\Delta}^2}{\varepsilon^{2 + \frac{2}{\beta - 1}}} \right)\right\}. 
\end{equation*}
Кроме того исследование вопроса о максимально допустимом уровне враждебного шума, возвращаемым безградиентным оракулом со значением целевой функции, является не менее важным, поскольку в некоторых приложениях (см. например, \cite{Bogolubsky_2016}) чем больше уровень враждебного шума $\Delta$, тем дешевле вызов безградиентного оракула: безградиентный оракул или оракул нулевого порядка в такой концепции шума принимает следующий вид: $f_\delta(x) = f(x) + \delta(x)$, $|\delta(x)| \leq \Delta$, т.е. оракул возвращает значение целевой функции с некоторым ограниченным шумом. Например, в случае повышенной гладкости функции при выполнении условия Поляка--Лоясиевича \eqref{eq:PL:PL} в работе \cite{Lobanov_2023} рассматриваются различные концепции с враждебным шумом. А также демонстрируется показательный результат преимущества рандомизированного алгоритма и эффективность использования ядерной аппроксимации \eqref{kernel_approx}. А в случае, когда функция не является гладкой, но гарантируется $M$-Липшицевость функции $f(x)$, такой что для всех $ x,y \in Q$:
\begin{equation*}
    |f(y) - f(x)| \leq M \|y - x\|_2,
\end{equation*}
существуют работы \cite{gasnikov2022power, dvinskikh2022noisy, Lobanov_2022, Kornilov_NIPS}, авторы которых предоставили оценки на максимально допустимый уровень шума $\Delta$ в различных настройках задачи, совпадающий с верхними границами, полученными в работе \cite{Risteski_2016} для класса выпуклых $M$-Липшицевых задач оптимизации.

Обзор современного состояния развития безградиентных методов для (сильно) выпуклых задач в условиях шума, см., например, в \cite{gasnikov2022randomized}. Выше был описан лишь один важный, но все же частный сюжет.

% \ag{\textbf{Александру Лобанову} Изложить суть ядерной оценки производной по направлению \cite{поляк1990оптимальные,bach2016highly,gasnikov2022randomized,akhavan2023gradient}. Причем тут стоит отметить и нижние оценки и известные сейчас верхние оценки, в том числе включающие условие Поляка--Лоясиевича ну и наши всевозможные уточнения и обобщения на случай разных концепций шума, разных режимов, например, \cite{lobanov2023zero,pasechnyuk2023upper} ... Но можно ограничиться только гладкими задачами (в том числе, повышенно гладкими). Это сократит объем материала. А на негладкие просто можно будет сослаться \cite{gasnikov2022power,gasnikov2022randomized,kornilov2023gradient}. Причем, по возможности, тут даже лучше ссылаться не на \cite{kornilov2023gradient}, а на ее правильную версию с ускорением, которую подали на НИПС2023.}

Настоящий доклад представляет пополненную расшифровку записи лекции 12 июля 2023 года А.В. Гасникова на Традиционной школе им. Б.Т. Поляка по оптимизации (\url{https://ssopt.org/}).

    %=========== the end of subsection 1.1: title "Подраздел" ==========
%=========== the end of section 1: Введение ==========

\printbibliography

% \newpage

%\appendix
%\input{4_appendix}

\end{document}